\documentclass[twocolumn,amsthm]{autart}

\usepackage{cite}
\usepackage{amsmath,amssymb,amsfonts,bbm, xspace,mathrsfs}
\usepackage{graphicx}
\usepackage{textcomp}
 \usepackage{bbm, xspace}
 \usepackage{mathrsfs}
 \usepackage{mathtools}
 \usepackage{color}
 \usepackage{epsfig}
 \usepackage[acronym]{glossaries}
 \usepackage{subfigure}
 \usepackage{cancel}
 \usepackage{empheq}
 \usepackage{placeins}
 
 \usepackage{enumitem}
 \usepackage{textgreek}
 \usepackage{mathbbol}
 \usepackage{stackengine}
\usepackage{dblfloatfix}
\usepackage{flushend}
\usepackage{algorithm}
\usepackage{algorithmicx}

\usepackage{fnpct}

\def\QEDclosed{\mbox{\rule[0pt]{1.3ex}{1.3ex}}} 
\def\QEDopen{{\setlength{\fboxsep}{0pt}\setlength{\fboxrule}{0.2pt}\fbox{\rule[0pt]{0pt}{1.3ex}\rule[0pt]{1.3ex}{0pt}}}}
\def\QED{\QEDclosed} 
\def\QEDhereeqn{\eqno\let\eqno\relax\let\leqno\relax\let\veqno\relax\hbox{\QED}}
\def\QEDopenhereeqn{\eqno\let\eqno\relax\let\leqno\relax\let\veqno\relax\hbox{\QEDopen}}

 \newcommand{\bs}{\boldsymbol}
 \newcommand{\mc}{\mathcal}
 \renewcommand{\emph}{\textit}
 
 \newcommand{\0}{\bs 0}
 \def\1{{\bs 1}}
 \def\argmin{\mathop{\rm argmin}}
  \def\argmax{\mathop{\rm argmax}}
 \newcommand{\col}{\mathrm{col}}
 \def\diag{\mathop{\hbox{\rm diag}}}

 \def\R{\mathbb{R}}
 
 \def\I{\mc{I}}
 \def\i{{i\in\mc{I}}}
  \def\J{{\mc{J}}}

  \def\jzero{{j\in\mc{J}_0}}

 \DeclareSymbolFontAlphabet{\mathbbm}{bbold}
 \DeclareSymbolFontAlphabet{\mathbb}{AMSb}%
 
 \stackMath
	 \newcommand\tsup[2][2]{%
	 	\def\useanchorwidth{T}%
	 	\ifnum#1>1%
	 	\stackon[-.5pt]{\tsup[\numexpr#1-1\relax]{#2}}{\scriptscriptstyle\sim}%
	 	\else%
	 	\stackon[.5pt]{#2}{\scriptscriptstyle\sim}%
	 	\fi%
	 }

\makeglossaries
\newacronym{KKT}{KKT}{Karush--Kuhn--Tucker}
\newacronym{ADMM}{ADMM}{alternating direction method of multipliers}
\newacronym{OPF}{OPF}{optimal power flow}
\newacronym{OPFP}{OPFP}{optimal power flow problem}
\newacronym{NUM}{NUM}{network utility maximization}
\newacronym{LMI}{LMI}{linear matrix inequality}
\newacronym{BMI}{BMI}{bilinear matrix inequality}
\newacronym{LM}{LM}{Lyapunov-Metzler}
\newacronym{SDP}{SDP}{semidefinite programming}
\newacronym{LTI}{LTI}{linear time invariant}
\newacronym{MJLS}{MJLS}{Markov jump linear system}
\newacronym{PID}{PID}{proportional-integral-derivative}
\newacronym{DT}{DT}{discrete-time}
\newacronym{CT}{CT}{continuous-time}
\newacronym{RM}{RM}{rank-minimization}

 \makeatletter \let\cl@part\relax \makeatother
 \usepackage{nohyperref}
 \usepackage[capitalize]{cleveref}
  
 \usepackage{letltxmacro}

 \crefname{thm}{Theorem}{Theorems}
 \crefname{lem}{Lemma}{Lemmas}
 \crefname{cor}{Corollary}{Corollaries}
 \crefname{rem}{Remark}{Remarks}
 \crefname{alg}{Algorithm}{Algorithm}
 \crefname{figure}{Figure}{Figures}
 \crefname{assumption}{Assumption}{Assumptions}
  \crefname{corollary}{Corollary}{Corollaries}
    \crefname{proposition}{Proposition}{Propositions}
    \crefname{problem}{Problem}{Problems}
     \crefname{example}{Example}{Examples}

 \crefname{thmlisti}{Theorem}{Theorem}
 \crefname{lemlisti}{Lemma}{Lemma}
 \crefname{asmlisti}{Assumption}{Assumptions}

 \newlist{thmlist}{enumerate}{1}
 \setlist[thmlist]{label=(\roman{thmlisti}), ref=\thethm(\roman{thmlisti}),noitemsep}
 \newlist{lemlist}{enumerate}{1}
 \setlist[lemlist]{label=(\roman{lemlisti}), ref=\thelem(\roman{lemlisti}),noitemsep}
 \newlist{asmlist}{enumerate}{1}
 \setlist[asmlist]{label=(\roman{asmlisti}), ref=\theassumption(\roman{asmlisti}),noitemsep,nosep,leftmargin=*} 

 \newtheorem{lemma}{Lemma}
 \newtheorem{theorem}{Theorem}
 \newtheorem{remark}{Remark}
 \newtheorem{assumption}{Assumption}
  \newtheorem{corollary}{Corollary}
    \newtheorem{proposition}{Proposition}
    \newtheorem{example}{Example}
\newtheorem{problem}{Problem}

\def\M{\mc{M}}
\def\Mp{\mc{M}_+}
\def\D{\mc{\noiseset}}
\def\C{\mc{C}}
\def\ms{{\hspace{-0.25em}}}

\def\Pit{\tilde P_i}
\def\Pjt{\tilde P_j}
\def\Pbs{\bs{P}_{\ms -i}}
\def\Pibs{\bs{\pi}_{ -i}}
\newcommand{\Metzler}{weight\@\xspace}  
\def\noise{w}
\def\dist{\psi}
\def\noiseset{W}
\def\distset{\Psi}

\def\Acl{A^{\textnormal{cl}}}
\def\Acli{{A^{\textnormal{cl}}_i}}
\def\Aclibar{{\bar A^{\textnormal{cl}}_i}}
\def\Aclj{{A^{\textnormal {cl}}_j}}

\def\Ccli{{C^{\textnormal {cl}}_i}}
\def\H{\mc{H}}

\usepackage{etoolbox}
\patchcmd{\smallmatrix}{\thickspace}{\kern.5em}{}{}
\newenvironment{smallbmatrix}
  { \left[\begin{smallmatrix}}
  {\end{smallmatrix}\right]}

\usepackage{caption}
\usepackage{booktabs}
\usepackage{multirow}

\makeatletter
\newcommand{\setpropositiontag}[1]{
  \let\oldtheproposition\theproposition
  \renewcommand{\theproposition}{#1}
  }
\makeatother

\def\QEDhereeqn{\eqno\let\eqno\relax\let\leqno\relax\let\veqno\relax\hbox{\QED}}
\def\QEDopenhereeqn{\eqno\let\eqno\relax\let\leqno\relax\let\veqno\relax\hbox{\QEDopen}}

\def\vmin{v_{\min}}

\newcommand{\mycomment}[1]{}

\newcommand{\redone}[1]{\mycomment{#1}}

\begin{document}

	\begin{frontmatter}
\title{
Data-driven stabilization of  switched and constrained linear systems\thanksref{footnoteinfo}
}

\thanks[footnoteinfo]{
	This work was partially supported by NWO (OMEGA 613.001.702), by the ERC (COSMOS 802348),  by NSF Award (CMMI-2044900), and by ETH Z\"urich funds.
}


		\author[aff0]{Mattia Bianchi},
		\author[aff1]{Sergio Grammatico},
		\author[aff2]{Jorge Cort\'es}

        \address[aff0]{Automatic Control Laboratory,  ETH Z\"urich, Switzerland}
		\address[aff1]{Delft Center for Systems and Control, Delft University of Technology, The Netherlands}
		\address[aff2]{Department of Mechanical and Aerospace Engineering, UC San Diego, California}

\maketitle

\begin{abstract} We consider the design of state feedback control laws for both the switching signal and the continuous input of an unknown switched linear system, given past noisy input-state trajectories measurements. 
Based on Lyapunov-Metzler inequalities and on a matrix S-lemma, we derive data-dependent bilinear programs, whose solution  directly returns a provably stabilizing controller and ensures $\H_2$ or $\H_{\infty}$ performance.
We further present  relaxations  that considerably reduce the computational cost, 
still without requiring stabilizability of any of the switching modes. 
 Finally, we showcase the flexibility of our approach on  the  constrained stabilization problem for an unknown perturbed linear system.
We validate our theoretical findings numerically, demonstrating the favourable trade-off between conservatism and tractability achieved by the proposed relaxations.

\end{abstract}

\end{frontmatter}



\section{Introduction}\label{sec:introduction}

\emph{Direct} data-driven control refers to the design of controllers based only on observed trajectories generated by an unknown dynamical system, without explicitly identifying its parameters\cite{DePersis_Formulas_TAC2020,Dorfler:Coulson:Markovsky:Bridging:TAC:2022,Krishnan:Pasqualetti:DirectvsIndirect:CDC:2021}. Bypassing the identification step  comes with important advantages. First, the computational complexity of system identification is mitigated.  
Second, uncertainty propagation is avoided.  In fact,
measurement noise can lead to inaccurate models:
while it is possible to bound the identification error  and  resort to robust controllers,  this two-step procedure is typically conservative.  Third, frequently control synthesis requires less information on the system (and hence, less data)  than identification of its dynamics \cite{VanWaardeEtal_TAC2020}. 

\emph{Literature review:} Direct data-driven control traces back to the work of Ziegler and Nichols \cite{Ziegler:Nichols:1942} on PID controllers. Classical methods are also virtual reference feedback tuning \cite{Campi:VRFeedbackTuning:AUT:2001}, iterative feedback tuning \cite{Hjalmarsson:IterativeFeedbackTuning:CS:1998}, reinforcement learning \cite{Bradtke:RLlqr:NEURIPS:1993} and extremum seeking \cite{KRSTIC_2000}. Later alternatives include  intelligent PID controllers \cite{Fliess:Join:IntelligentPID:IFAC2009} and  model-free adaptive control \cite{Hou:Jin:ModelFreeAdaptiveControl:TNN:2011}; methods based on Willems' fundamental lemma are currently  enjoying renewed popularity  as well \cite{Coulson:Lygeros:Dorfler:DeePC:ECC2019,DePersis_Formulas_TAC2020, MARKOVSKY2023_persistency_AUT,Berberich:Kohler:Allgower:DataDrivenMPC:TAC2021,RuedaEscobedo:Schiffer:VolterraSystems:CDC2020,Allibhoy:Cortes:MPC:LCSS:2021,BRESCHI2023110961}.

In this paper, we adopt instead a recent robust-control approach, 
whose strength is to provide non-asymptotic theoretical guarantees  \cite{Dai_Semialgebraic_LCSS2020,VanWaardeEtal_TAC2020,BISOFFI_petersen_AUT2022}.
In simple terms, when the true plant is not known, the goal is to find a controller that provably ensures closed-loop stability (or performance/optimality \cite{VanWaardeEtal_Noisy_TAC2022}, safety \cite{Ahmadi:Israel:Topcu:SafeDifferentialInclusions:TAC2020}) for \emph{all the systems compatible} with 
(a) a few, finite, open-loop measured trajectories, possibly  corrupted by noise, and
(b) prior knowledge on the model class (e.g., \gls{LTI} systems \cite{VanWaardeEtal_TAC2020}, polynomial dynamics  \cite{Dai_Semialgebraic_LCSS2020}) and noise bounds.
\color{black}
The controller is typically built by solving a data-dependent optimization problem, for instance a \gls{LMI} \cite{VanWaardeEtal_TAC2020} or a polynomial program  \cite{Dai_Semialgebraic_LCSS2020}. In this stream of literature, much effort is devoted to providing \emph{tractable} conditions for stabilization of several classes of nonlinear systems and under various noise bounds\cite{Strasser:Berberich:Allgower:CDC2021,DePersis:Rotulo:Tesi:Cancellation:arXiv2022,Martin:Allgower:PolynomialApproximations:arXiv2022,Dai_Switched_CDC2018,ROTULO2022_switched_AUT}. 
In particular, the works \cite{Dai_Switched_CDC2018,DAI_switched_AUT2022,ROTULO2022_switched_AUT,Eising_switched_CDC2022,eising2023_switched} focus on the data-driven design of the \emph{continuous} input for \emph{switching} linear systems \cite{Liberzon2003}, where the dynamics  switches \color{black} 
%
%
freely among a set of \gls{LTI} plants (also called ``modes'').  In  \cite{Dai_Switched_CDC2018,DAI_switched_AUT2022}, the switching signal is arbitrary but measured. The authors of \cite{ROTULO2022_switched_AUT,Eising_switched_CDC2022,eising2023_switched} study the case of an unknown  switching signal, with stability  ensured under a sufficiently large dwell time.
%
%
Both cases require  stabilizability of each subsystem. 

On the contrary, in this paper, we   consider \emph{switched} linear systems --where the active mode is chosen by the controller. In particular, we are interested in ``stabilization by switching'', namely controlling the
system by opportunely choosing the \emph{discrete} switching signal: famously, this can be achieved even when all the individual modes are unstable \cite{Liberzon2003}. Not only switched dynamics naturally arise in prominent engineering applications (robotics, embedded systems, traffic control, power electronics \cite{Baldi:PowerConverters:HS:2018,Deaecto:Networked:TAC2015}), but even for a single plant switching  among different compensators can achieve otherwise impossible control goals \cite[Ch.~9]{Blanchini:Miani:SetTheoretic:2015}, \cite{Breschi_switching_IJC_2020}.
For this reason, stabilization by switching has been extensively investigated (for  a known system)~\cite{Morse:LogicBasedSwitching:1995,Liberzon2003,Fiacchini:Jungers:AUT:2014,Egidio:Deaecto:TAC:2021}. The challenge remains finding good trade-offs between complexity and conservatism: for instance, tight conditions  for  stabilizability of \gls{DT} switched linear systems are known, but computationally prohibitive  \cite{Fiacchini:Jungers:AUT:2014}. In this spirit,  the main tool is the \gls{LM} inequalities, introduced by Geromel and Colaneri \cite{GeromelColaneri_SIAM2006,GeromelColaneri_IJC2006},  that provide sufficient conditions for stability in the form of a \gls{BMI}.

In contrast, \emph{data-driven} stabilization of linear dynamics with controlled switching is essentially unexplored. To the best of our knowledge, the only work to address this problem (under additional dwell-time constraints) is \cite{Kundu:DataDrivenSwitchingLogic:arXiv:2020}. Yet, the design is particularly restrictive (e.g., a specific class of linear systems is considered, and  at least one of the modes must be stable) and the assumption on the \emph{noiseless} data implies that each subsystem can be uniquely identified. \cite{Zhang:Switched:IJSS:2019} also studies  switched systems, but rather considers a finite-horizon control task via a reinforcement learning approach. 

\emph{Contributions:} In this paper, we consider the direct data-driven stabilization of an unknown continuous-time switched linear system. We start by observations of finite input-state-derivative data generated by the system; 
differently from \cite{Kundu:DataDrivenSwitchingLogic:arXiv:2020}, the measurements are subject to noise, obeying an energy-type bound.
We provide a method to design a  state-feedback controller (for the discrete switching signal, and for the continuous input if present) that provably stabilizes the unknown dynamics, with guaranteed performance. 
Our results are complementary to those in \cite{Dai_Switched_CDC2018,DAI_switched_AUT2022,Eising_switched_CDC2022,ROTULO2022_switched_AUT}  about stabilization of systems with  uncontrolled switching. We build upon \gls{LM} inequalities  \cite{GeromelColaneri_SIAM2006}, and on a matrix S-lemma, developed in \cite{VanWaardeEtal_Noisy_TAC2022} for data-driven control of linear discrete-time systems, but that found numerous other applications in the field \cite{Guo:DePersis:Tesi:PolynomialNoisy:TAC2021,VanWaarde_QMI_2023}. Compared to \cite{GeromelColaneri_SIAM2006}, we deal with the \emph{set} of systems unfalsified by the data, rather than with a well-known model. Unlike \cite{VanWaardeEtal_Noisy_TAC2022}, we consider continuous-time switched systems, resulting in a substantially different analysis.   The technical novelties of our work are summarized as follows:
\begin{itemize}[leftmargin=*] 
    \item \emph{Data-driven \gls{LM} inequalities}: 
    We derive a \gls{BMI}, dependent only on measurements and prior structural knowledge, whose solution directly supplies a stabilizing controller. The condition is obtained by applying the S-procedure
    \cite{VanWaardeEtal_Noisy_TAC2022} to a dual version of the \gls{LM} inequalities of \cite{GeromelColaneri_SIAM2006}.
    Under a Slater's condition (as in \cite{VanWaardeEtal_Noisy_TAC2022}), our formulation is nonconservative, i.e., it is solvable if and only if there exists a  solution to the \gls{LM} inequalities common to all the systems that are compatible with the data.
    Crucially, although this set of systems can be unbounded, we prove that it is not restrictive to constraint some variables to be strictly positive --in turn allowing for their inversion, which is fundamental for the dualization of the \gls{LM} inequalities
    (Section~\ref{sec:main}); 
    \item \emph{Relaxations}: As for the model-based case, our data-driven \gls{LM} inequalities are nonconvex and challenging to solve. Thus, we propose two relaxations that, at the price of some conservatism, considerably reduce the computational burden.  In particular, the second condition is an \gls{LMI} when a scalar variable is fixed (hence, it can be efficiently solved via \gls{SDP} and line-search), and directly improves on that  in \cite{GeromelColaneri_SIAM2006}, by being both easier to solve and  less restrictive. 
    Importantly, we also provide sufficient conditions  for solvability of our relaxed inequalities, directly in terms of the system properties (Section~\ref{sec:relaxations});
    \item \emph{Performance}: We derive data-dependent \gls{LM} inequalities to solve the  $\mc{H}_2$ and $\mc{H}_\infty$ control problems for unknown switched linear systems (Section~\ref{sec:Extensions:H});
    \item \emph{Safe stabilization}: We  show that the tools developed in this paper are of interest even when the original system is \gls{LTI}. In particular, we design, solely based on data, a switched controller and a robustly invariant set that  guarantee not only robust satisfaction of state (and input) constraints, but also --differently from the results in \cite{Bisoffi:DePersis:Tesi:RobustInvariance:arXiv2020,Luppi:Bisoffi:DePersis:Tesi:SafePolynomial:arXiv2021}-- closed-loop asymptotic stability (Section~\ref{sec:Constrained}). 
    \end{itemize}

Some background on \gls{LM} inequalities is provided in Appendix~\ref{app:LM}, where we also prove novel results on stabilization of switched systems in the case of sliding motions: this is an extension of general interest, that we then  leverage in our data-driven formulation. Finally, in Section~\ref{sec:Numerics} we illustrate our data-driven methods via numerical examples, with  attention to the tractability/conservatism trade-off.
    

\emph{Notation:} We  use the compact notation $\tilde P\coloneqq P^{-1}$ for the inverse of a matrix $P$. We denote $\Delta\coloneqq \{\lambda \in \R^N_{\geq 0} \mid \1_N^\top \lambda=1\}$ and $\Delta_+\coloneqq \{\lambda \in \R^N_{> 0} \mid \1_N^\top \lambda=1\}$; let $\mc{S}$ ($\mc{S}_+$) be the set of matrices whose columns belong to $\Delta$ ($\Delta_+$). We define the sets of Metzler matrices
$ \M  \coloneqq \{ \Pi  \in  \R^{N\times N} \mid \pi_{i,j}\geq 0 \ \forall i \neq j,  \sum_{i=1}^N  \pi_{i,j}=0 \  \forall j \}$ and 
$\Mp  \coloneqq \{ \Pi  \in  \R^{N\times N} \mid \pi_{i,j}> 0 \ \forall i \neq j, \ \textstyle \sum_{i=1}^N\pi_{i,j} = 0 \ \forall j \}$,
where $\pi_{i,j}$ is the element of $\Pi$ in row $i$ and column $j$. For a signal $s:\R_{\geq 0}\rightarrow \R^m$, its $2$-norm is $\|s \|_2 = \left(\int_{t = 0}^ \infty s(t)^\top s(t) \text{d}t\right) ^{\frac{1}{2}}$; we denote by $\mc{L}_2$ the set of signals with bounded $2$-norm. $e_i \in \R^m$ is the $i$-th column of the identity matrix $I$ of appropriate dimension $m$. $\delta(t)$ denotes the continuous-time unitary impulse. $P \succ 0$ ($\succcurlyeq 0$) denotes a symmetric positive (semi-)definite matrix; $\text{\textlambda}_{\min}(P) $  and $\text{\textlambda}_{\max}(P)$ denote the minimum and maximum eigenvalue of $P$. We may replace the blocks of a matrix that can be deduced by symmetry with the shorthand notation ``$\star$''. 
$\|\cdot\|_\textnormal{F}$ denotes the Frobenius matrix norm.


\begin{lemma}[Matrix S-lemma {\cite[Th.~9]{VanWaardeEtal_Noisy_TAC2022}}]\label{lem:S} Let $G,H\in \R^{(k+n)\times (k+n)}$ be symmetric. Consider the following:
\begin{itemize}
    \item[(a)] $\exists \alpha \in \R_{\geq 0} \text{ such that } G-\alpha H \succcurlyeq 0$; 
    \item[(b)]\!\!$\begin{bmatrix}
          I \\ Z
    \end{bmatrix}^{\! \top}
    \!G \begin{bmatrix}
          I \\ Z
    \end{bmatrix} \! \succcurlyeq \! 0, \forall Z \in \R^{n\times k} \text{ s.t. } \begin{bmatrix}
          I \\ Z
    \end{bmatrix}^\top H \begin{bmatrix}
          I \\ Z
    \end{bmatrix}\!\succcurlyeq 0$.
\end{itemize}
%
%
Then $\textnormal{(a)}\Rightarrow\textnormal{(b)}$; if in addition  $\exists \bar Z\in \R^{n\times k}$ such that $ \begin{bmatrix}
          I &  \bar Z^\top 
    \end{bmatrix} H \begin{bmatrix}
          I & \bar Z^\top 
    \end{bmatrix}^\top \succ 0$,  then also $\textnormal{(b)}\Rightarrow\textnormal{(a)}$. \hfill $\square$
\end{lemma}

\section{Problem statement}\label{sec:prob}
We consider a  switched linear system 
\begin{align}\label{eq:system}
    \dot{x} =\bar{A}_\sigma x+ \bar B_\sigma u,
\end{align}
where $x\in \R^n$, $u\in \R^m$ is a  controlled continuous input, $\sigma \in \mc{I} \coloneqq \{1,2,\dots,N\}$ is a  controlled discrete input (switching signal), and $\{\bar A_i\}_{\i}$, $\{ \bar B_i \} _{\i}$ are the system matrices. Here the bar indicates the \emph{true} system matrices, which are  unknown.

Instead, we assume that some experimental data are available, generated by applying  inputs $u$ and $\sigma$, measuring the state $x$ and obtaining an estimate of the derivative, $\dot x +\noise$, for an unknown disturbance $\noise$. In particular,  $T_i \geq  0$ samples have been collected for each subsystem $\i$ (at non switching instants $ \{t^i_k\}_{k=1,\dots,T_i } $, not necessarily from a unique trajectory), and organized in the following matrices:
\begin{align*}
    X_i & \coloneqq \begin{bmatrix} x(t_1^ i) & x(t_2^ i) & \dots & x(t_{T_i}^ i) \end{bmatrix} 
    \\ 
    \dot X_i & \coloneqq  \begin{bmatrix} \dot x(t_1^ i)\!+\! \noise(t_1^ i) & \dot x (t_2^ i)\!+\!\noise(t_2^ i) & \dots & \dot x(t_{T_i}^ i)\!+\! \noise(t_{T_i}^ i) \end{bmatrix} 
    \\ 
    \bar \noiseset_i & \coloneqq  \begin{bmatrix} \noise(t_1^ i) & \noise(t_2^ i) & \dots & \noise(t_{T_i}^ i) \end{bmatrix}
    \\
    U_i & \coloneqq  \begin{bmatrix} u(t_1^ i) & u(t_2^ i) & \dots & u(t_{T_i}^ i) \end{bmatrix},
\end{align*}
that satisfy 
\begin{align}\label{eq:data}
    \dot X_i = \bar A_i X_i + \bar B_i U_i + \bar \noiseset_i,
\end{align}
where the matrix $\bar{\noiseset}_i$ is \emph{unknown}.
\begin{assumption}[Disturbance model]\label{asm:disturbancemodel}
For each $\i$, $\bar \noiseset_i \in \D_i$, where 
\begin{align*}
\D_i \coloneqq \Biggl\{  \noiseset_i \in \R^{n \times T_i} \bigl\vert 
    \begin{bmatrix}
    I \\  \noiseset_i^\top 
    \end{bmatrix}^\top 
    \underbrace{ \begin{bmatrix}
        \Phi_{1,1}^i & \Phi_{1,2}^i \\[0.5em]
        {\Phi_{1,2}^i}^\top & \Phi_{2,2}^i
    \end{bmatrix} }_{\coloneqq \Phi^i}
       \begin{bmatrix}
    I \\  \noiseset_i^\top 
    \end{bmatrix} \succcurlyeq 0 \Biggr\} 
\end{align*}
for known  $\Phi^{i}_{1,1}={\Phi_{1,1}^{i}}^\top$, $\Phi_{1,2}^i$, $\Phi^{i}_{2,2}={\Phi^{i}_{2,2}}^\top \prec 0$. \hfill $\square$
\end{assumption}

By replacing \eqref{eq:data} in \cref{asm:disturbancemodel}, we infer that, for any $\i$, a pair of matrices $(A_i,B_i)$ is consistent with the experiments (i.e., can explain the data) if and only if 
\begin{align}\label{eq:compatibility_ineq}
    \begin{bmatrix}
    I \\ A_i ^\top \\ B_i^\top 
    \end{bmatrix}^{\ms \top} \ms
    \underbrace{\begin{bmatrix}
      I & \dot X_i \\ 
      0 & - X_i \\ 
      0 & - U_i 
    \end{bmatrix}  \ms
     \begin{bmatrix}
        \Phi_{1,1}^i & \Phi_{1,2}^i \\[0.5em]
        {\Phi_{1,2}^i}^\top & \Phi_{2,2}^i
    \end{bmatrix}  \ms
       \begin{bmatrix}
      I & \dot X_i \\ 
      0 & - X_i \\ 
      0 & - U_i 
    \end{bmatrix}^{ \ms \top}}_{\coloneqq H_i} \ms
  \begin{bmatrix}
    I \\ A_i ^\top \\ B_i^\top 
    \end{bmatrix} \ms \succcurlyeq 0, 
\end{align}
where the matrix $H_i$ depends on known quantities only. Let us define, for each $\i$, the compatibility set
\begin{align}
   \ms\ms \C_i & \coloneqq \{(A_i,B_i) \mid  \exists \noiseset_i \in \D_i : \dot X_i =A_i X_i +B_i U_i +\noiseset_i \} \ms   \nonumber  \\
       & \,=\{ (A_i,B_i) \mid \  \eqref{eq:compatibility_ineq} \text{ holds} \}. \label{eq:compatibility_definition} 
\end{align}
Clearly, $(\bar A_i, \bar B_i)\in \C_i $; however, the true system matrices cannot  be discerned via the available data. 
Thus, to ensure stability of the true system, we need to find a controller that would stabilize every plausible plant.

\begin{problem}\label{prob:main} Under \cref{asm:disturbancemodel}, find state-feedback laws for the input $u$ and the switching signal $\sigma$ that stabilize the system in  \eqref{eq:system} (with $\bar{A}_\sigma$, $\bar B_\sigma$ replaced by $A_\sigma$, $B_\sigma$) for any set of matrices compatible with the measured data in \eqref{eq:data}, i.e., for all $(A_i,B_i)_{\i} \in \C \coloneqq \C_1 \times \C_2 \times \dots \times  \C_N$, with $\mc{C}_i$ given in~\eqref{eq:compatibility_definition}.  \hfill $\square$
\end{problem}

\begin{remark}[On the disturbance model] 
\label{rem:alternative_disturbance}
The unknown-but-bounded model in \cref{asm:disturbancemodel} is standard in the recent literature \cite[Asm.~1]{Strasser:Berberich:Allgower:CDC2021}, \cite[Asm.~1]{VanWaardeEtal_Noisy_TAC2022}, \cite[Asm. 2]{depersis2022eventtriggered}, \cite[Eq. 6]{eising2023_switched}. It is an energy-like bound. The matrix $\Phi^i$ is chosen by the designer to encode (or overapproximate) the prior bounds available on the disturbance, for instance bounds on the noise-to-signal ratio,  sample covariance or individual samples (see \cite{VanWaarde_QMI_2023} for a detailed discussion).
For instance, if the prior  on the disturbance is given by instantaneous bounds on the norm of $\noise_i$, i.e., $\sup_t \|\noise_i(t)\|^2\leq \overline{\noise}_i$, then  \cref{asm:disturbancemodel} holds by choosing $\Phi^i_{1,1} = T_i \overline \noise_i I$, $\Phi^{1}_{1,2}=0$, $\Phi^i_{2,2} = -I$
(a tighter upper bound on the compatibility sets in the form \eqref{eq:compatibility_ineq}, for some data-dependent matrix $H_i$, could also be computed as in \cite[Eq.~18]{Luppi:Bisoffi:DePersis:Tesi:SafePolynomial:arXiv2021}, by solving a data-dependent program). 
Moreover, \cref{asm:disturbancemodel} also comprises the case of zero disturbance, by choosing $\Phi_{1,1}^i$ and $\Phi_{1,2}^i$ as zero matrices (in turn, $\C$ can be a singleton if the data are rich enough, i.e., the system \eqref{eq:system} is identifiable). 
Finally, we interpret $\noise$ as  noise on the derivative estimate. First, this estimate can be obtained without actually measuring the derivative, for example by Euler approximation (with guaranteed error bound satisfying \cref{asm:disturbancemodel} \cite{depersis2022eventtriggered}); in fact, one can also obtain data of the form \eqref{eq:data} without numerical derivation, via an integral form, see \cite[App.~A]{depersis2022eventtriggered}. Second, it is actually irrelevant how $\noise$ is generated, as long as \eqref{eq:data} and \cref{asm:disturbancemodel} hold. On this basis,
we can also account for process disturbance (as in \cref{sec:Extensions:H}), or inaccuracy on the off-line state measurement $X_i$, see  \cite[Rem.~2] {depersis2022eventtriggered} (although we assume exact on-line state measurement for the implementation of our controllers, as in  the recent data-driven literature \cite{DePersis_Formulas_TAC2020,Dai_Semialgebraic_LCSS2020,depersis2022eventtriggered}; the case of online noisy measurements remains challenging, even for a perfectly known system). \color{black} \hfill$\square$
\end{remark}

\begin{remark}[State measurement]
We assume that measurements are collected about the state, as in virtually all the data-driven literature that takes the robust-control approach \cite{DePersis:Rotulo:Tesi:Cancellation:arXiv2022,Dai_Semialgebraic_LCSS2020,Strasser:Berberich:Allgower:CDC2021}. 
Dealing with input-output data is an important open problem, partially settled only for \emph{discrete-time} \emph{linear} systems \cite{DePersis_Formulas_TAC2020,VanWaarde_QMI_2023}.
\hfill $\square$
\end{remark}





 \section{Data-driven stabilization}\label{sec:main}
 
 For the continuous input $u$, we restrict our attention to linear controllers of the form $u=K_\sigma x$, where $\{K_i\}_{\i}$ are feedback gains to be designed. For brevity, we define \begin{align}\label{eq:Acl} \Acl_i\coloneqq A_i+B_i K_i.
\end{align}
To solve \cref{prob:main} in a data-driven fashion, we aim at leveraging \cref{prop:LMstability} in Appendix~\ref{app:LM}. In particular, if we can find  $\{K_i\}_{\i}$,  $\{P_i \succ 0\}_\i$, $Q\succcurlyeq 0$ and $\Pi \in \M $ such that for all $\i$,  for all  $(A_i,B_i)\in \C_i$,
%
%
\begin{align}\label{eq:LM1} 
\Acli ^\top P_i + P_i \Acli + \textstyle \sum_{j\in \mc{I}}  \pi_{j,i} P_j+Q \prec  0,
\end{align}
 then \cref{prop:LMstability} ensures that the controller 
 \begin{align}\label{eq:control}
     \sigma(x) = \min \left\{ \textstyle \argmin_{\i} \ x^\top P_i x \right\}, \ u(x) = K_{\sigma(x)} x,
 \end{align}
 \begin{figure*}[!h]
\begingroup%
\thinmuskip=0mu plus 1mu
\medmuskip=0mu plus 2mu
\thickmuskip=1mu plus 3mu
\begin{align}
\label{eq:main}
\tag{C1}
 \begin{bmatrix}\\[-0.8em]
        \tilde Q & 0 & \tilde P_i & 0 & 0
        \\
        0& 
        \tilde{\bs{\pi}}_{-i}\tilde {\bs{P}}_{\ms -i} & (\1 \otimes I)\tilde P_i & 0& 0 
        \\
        \tilde P_i &  \Pit (\1_{N-1} \otimes I_n)^\top   &  -\pi_{i,i}\Pit  & -\Pit & -L_i^\top  \\
      0 & 0&  -\Pit & 0 & 0 \\
      0& 0&  -L_i & 0 & 0
    \end{bmatrix}
- \alpha_i \begin{bmatrix}
      0& 0 \\
      0 & 0 \\
      I & \dot X_i \\ 
      0 & - X_i \\ 
      0 & - U_i 
    \end{bmatrix}  
       \begin{bmatrix}
        \Phi_{1,1}^i & \Phi_{1,2}^i \\[0.5em]
        {\Phi_{1,2}^i}^\top & \Phi_{2,2}^i
    \end{bmatrix} 
       \begin{bmatrix}
      0& 0 \\
      0 & 0 \\
      I & \dot X_i \\ 
      0 & - X_i \\ 
      0 & - U_i 
    \end{bmatrix}^\top \succcurlyeq 0 
\end{align}
\hrule
\endgroup%
\end{figure*}
 asymptotically stabilizes \eqref{eq:system}, with guaranteed performance $\int_{0}^{\infty} \ x^\top Q x \ \text{d}t < \min_{\i} x(0)^\top P_i x(0)$. Note that we look for a  \emph{common} \Metzler matrix $\Pi$ for all the systems in the compatibility set $\C$, akin to most methods in the  literature that seek a common Lyapunov function \cite{VanWaardeEtal_Noisy_TAC2022,Strasser:Berberich:Allgower:CDC2021} --instead, $K_i$ and $P_i$ \emph{must} be common to all systems  in $\C_i$ for  \eqref{eq:control} to be implementable.

 In place of \eqref{eq:LM1}, we consider the following related problem, more suitable for our data-driven formulation because of the strict conditions on $Q$ and $\Pi$, as shown later.


 \begin{problem}\label{prob:LM}
 Find $\{K_i\}_\i$, $\{P_i\succ 0\}_{\i}$, $Q \succ 0$, $\Pi \in \Mp$ such that for all  $\i$, for all $(A_i,B_i)\in \C_i$,
 \begin{align}\label{eq:LM2}
\Acli^\top P_i + P_i \Acli + \textstyle \sum_{j\in \mc{I}} \pi_{j,i} P_j+Q \preccurlyeq 0. \hspace{-0.3em}
\end{align} 
\hfill $\square$
 \end{problem}
Note that a solution to \eqref{eq:LM2} immediately provides a solution to \eqref{eq:LM1} (e.g., with $Q=0$ in \eqref{eq:LM1}). Thus implementing \eqref{eq:control} with matrices $\{P_i,K_i\}_\i$ that satisfy \eqref{eq:LM2} does stabilize \eqref{eq:system} and assures that
\begin{align}\label{eq:performance1}
 \int_{0}^{\infty} \ x^\top Q x \ \text{d}t \leq \min_{\i} x(0)^\top P_i x(0).
\end{align}
Conversely, we wonder if feasibility of \eqref{eq:LM1} implies feasibility of \eqref{eq:LM2}. When the compatibility set $\C$ is a singleton (i.e., if the real system is known), or more generally when $\C$ is compact, this is immediately true, due to the strict inequality in \eqref{eq:LM1}. 
The following result shows that \eqref{eq:LM2} is nonrestrictive (i.e., \eqref{eq:LM1} can be bounded away from zero) even in case of unbounded compatibility sets $\C_i$'s.

\begin{lemma}[Equivalent \gls{LM} conditions]
\label{lem:strictnonstrict}
The system of inequalities in \eqref{eq:LM1} is feasible if and only if the system of inequalities in \eqref{eq:LM2} is feasible. \hfill $\square$
\end{lemma}

The proof of this result is given in Appendix~\ref{app:lem:strictnonstrict}.


To impose \eqref{eq:LM2} for all systems in the compatibility set, we  use the S-procedure \cite{VanWaardeEtal_Noisy_TAC2022}. First, we need to ``dualize'' \eqref{eq:LM2} (to have $A_i$,$ B_i$ as factors on the left, and their transpose on the right, as in \eqref{eq:compatibility_ineq}). By left- an right-multiplying both sides of \eqref{eq:LM2} by $\Pit=P_i^{-1}\succ 0$, we obtain the equivalent inequality
\begin{align}\label{eq:dualform}
    \begin{bmatrix}
    I \\ A_i ^\top \\ B_i^\top
    \end{bmatrix}^{\! \! \top} \ms 
       \underbrace{ \ms  \ms \begin{bmatrix}
        -\Pit (Q+\underset{j\in\mc{I}}{\textstyle \sum} \pi_{j,i} P_j ) \tilde P_i & -\Pit & -\Pit K_i^\top \\[1em]
        -\Pit & 0 & 0 \\
        -K_i \Pit & 0 & 0
    \end{bmatrix} \ms \ms}_{\coloneqq G_i} \ms \!  
  \begin{bmatrix}
    I \\ A_i ^\top \\ B_i^\top 
    \end{bmatrix} \ms \! \! \succcurlyeq 0.
\end{align}
Then  \cref{lem:S} shows that \eqref{eq:dualform} holds for all $(A_i,B_i)$ satisfying \eqref{eq:compatibility_ineq} if there is a scalar $\alpha_i\geq 0$ such that 
\begin{equation}\label{eq:Slemmacompact}
    G_i - \alpha_i H_i \succcurlyeq 0;
\end{equation}
necessity also holds under the mild Slater's condition
\begin{align}\label{eq:Slaters}
\exists Z_i \in \R^{(n+m) \times n} \text{ s.t. } \begin{bmatrix}
  I & Z_i^\top 
\end{bmatrix}^\top 
H_i
\begin{bmatrix}
  I &  Z_i^\top 
\end{bmatrix}^\top \succ 0.
\end{align}
Note that, in the top-left corner of  $G_i$ in \eqref{eq:dualform},  $\tilde P_i$ and $P_j$ and $\pi_{i,j}$ are all variables. To  cope with this complication --that arises due to the coupling in \eqref{eq:LM_original} and is absent in problems for \gls{LTI} systems \cite{VanWaardeEtal_Noisy_TAC2022}-- we exploit \cref{lem:strictnonstrict} and a Schur complement argument.
The following is our main result of this section; let us define, throughout the paper,\color{black}
\begin{subequations}
\begin{align}
    \Pbs \coloneqq \diag ((P_{j})_{j\in \I \backslash \{i\}}), \\ \Pibs \coloneqq \diag ((\pi_{j,i}I_n)_{j\in \I \backslash \{i\}}).
\end{align}
\end{subequations}
\begin{theorem}[Data-driven \gls{LM} inequalities] \label{th:main}
For all $\i$, let the Slater's condition in \eqref{eq:Slaters} hold. Then, $(\{P_i\succ 0,K_i\}_\i,\Pi \in \Mp,Q\succ 0)$ solve \cref{prob:LM} if and only if there exist scalars $\{\alpha_i\geq 0 \}_{\i}$ such that
$(P_i,L_i\coloneqq K_i \tilde P_i \}_\i,\Pi, Q,\{\alpha_i\}_{\i})$ verify the inequality \eqref{eq:main} on top of this page, for all $\i$. \hfill $\square$
\end{theorem}
\begin{pf}
By taking the Schur complement with respect to the $2$-by-$2$ block upper-left matrix in \eqref{eq:main}, and by definition of $L_i$, we retrieve \eqref{eq:Slemmacompact}.
\hfill $\blacksquare$
\end{pf}

Sufficiency in \cref{th:main} holds also without the assumption in \eqref{eq:Slaters}. Hence, since the unknown real system belongs to the consistency set, we have the following. 
\begin{corollary}[Data-driven stabilization]\label{cor:main1}
Assume that \sloppy $(\{ \tilde P_i\succ 0\}_\i, \{\alpha_i \geq 0\}_\i,\{L_i \}_\i, \tilde Q\succ  0,\Pi\in \Mp)$ satisfy \eqref{eq:main}, for all $\i$. Then the controller  in \eqref{eq:control} with $K_i=L_i  P_i$ globally asymptotically stabilizes the switched system in \eqref{eq:system}. Furthermore, \eqref{eq:performance1} holds true. 
\hfill $\square$
\end{corollary}


% ...

Some technical remarks are in order.

\begin{remark}[Batch sizes]
\label{rem:batchsizes} The number of variables and the dimension in \eqref{eq:main} grows with  $N$, but does not depend on the dataset lengths $\{T_i\}_\i$, a convenient feature inherited from \cite{VanWaardeEtal_Noisy_TAC2022}. Depending on the disturbance model in \cref{asm:disturbancemodel}, a larger batch size can result in a smaller set $\mc{C}_i$, thus favoring solvability of \eqref{eq:main}.  \hfill $\square$
\end{remark}

\begin{remark}[Nonnegative \Metzler matrix]\label{rem:nonpositive Metzlermatrix}
Using \eqref{eq:LM2} in place of \eqref{eq:LM1} is crucial for \cref{th:main} as it allows inversion of  $\Pibs$ and $Q$. However, if $\Pi$ is given and known a priori, it is easy to accommodate for $\Pi \in \M$ instead of $\Pi \in \Mp$ (i.e., $\pi_{i,j}$ might be zero --note that stability of the closed-loop is not affected; this case will be useful in the following, see \cref{ex:LTI,ex:Markov}): it is enough to remove in $\Pibs$, $\Pbs$ any  block $j$ such that $\pi_{j,i}=0$ and cancel the corresponding rows and columns in \eqref{eq:main}.  \hfill $\square$ 
\end{remark}

\begin{remark}[Unbounded $\mc{C}$ and data collection]
If $\C_i$ is quadratically stabilizable\footnote{Namely, if there are $K_i$, $P_i\succ 0$ such that $\Acli ^\top P_i + P_i \Acli \prec 0$ for all $(A_i,B_i)\in\C_i$;  this  condition can be easily checked from data  \cite[Th.~14]{VanWaardeEtal_Noisy_TAC2022}.}, then \cref{prob:main} admits a trivial solution with $\sigma(t) \equiv i$; yet, the system in \eqref{eq:LM1} might still be unfeasible when $\C_j$ is unbounded, for some $j \neq i$. In fact, the proof of \cref{lem:strictnonstrict} shows that if  \eqref{eq:LM2} admits a solution, then, for all $j \in \mc{I}$
\begin{align}\label{eq:rankcondition}
    \operatorname{rank} (X_j) = n.
\end{align}
%
%
This necessary condition on the data is well known for
the case $N=1$ (i.e., \gls{LTI} systems) \cite[Lem.~15]{VanWaardeEtal_TAC2020}. 
Unfortunately, except for the case of linear noiseless controllable dynamics \cite{Willems_2005}, it is a vastly open problem how to design the  inputs of a system to guarantee that the obtained trajectories are ``persistently exciting'', e.g., satisfy \eqref{eq:rankcondition}  \cite[Sec.~V]{DePersis_Formulas_TAC2020}. Luckily, the rank condition can be easily checked a posteriori: if the data collected do not verify \eqref{eq:rankcondition} for some $j$, then that mode must be ignored in the design (by removing $j$ from $\I$). This simple expedient automatically avoids pathological unfeasibility of \eqref{eq:main} of the kind described above. Note that, even under \eqref{eq:rankcondition}, $\mc{C}_j$ can still be unbounded. \hfill $\square$
\end{remark}

\begin{remark}[Linear convergence]
\label{rem:rate}The proof of \cref{prop:LMstability} in Appendix~\ref{app:LM} shows that  stability in \cref{cor:main1} is 
exponential, with rate $\text{\textlambda}_{\min}(Q) / \max_{\i} \{\text{\textlambda}_{\max} (P_i)\}$.  \hfill $\square$
\end{remark}

\begin{remark}[Computational cost of \eqref{eq:main}] 
\label{rem:computationalcost} The inequalities \eqref{eq:main} depend on experimental data and prior on the disturbance only; hence, their solution provides a direct data-driven criterion to seek a stabilizing controller for \eqref{eq:system}.  The problem is nonconvex; it is cast as a \gls{BMI} in the variables $(\{\Pit,L_i\}_\i,\tilde Q, \Pi, \{\tilde \pi_{i,j} \}_{i,j \in \mc{I}}$), via the additional  constraints 
\begin{align}\label{eq:additionalconstraints}
  \pi_{i,j} \tilde \pi_{i,j} =1, \qquad \forall \i,\forall j\in \mc{I} \backslash \{i\},
\end{align} that enforce  $\Pi \in \Mp$.
Due to their ubiquity in control systems, significant effort has been devoted to the development of algorithms \cite{Quoc_BMI}  and software \cite{PENBMI} for the solution of \glspl{BMI}, typically resorting to  sequential convex (often, \gls{SDP}) relaxations. These tools can be readily used to solve \eqref{eq:main}. 
Yet, the problem remains computationally expensive and poorly scalable. Let us recall that the model-based condition in \cite{GeromelColaneri_SIAM2006} (see \cref{prop:LMstability}) is also a  \gls{BMI}; more recently,  bilinear conditions have been also proposed in  the context of data-driven control, e.g., \cite{Dai_Semialgebraic_LCSS2020,Dai_Switched_CDC2018,Luppi:Bisoffi:DePersis:Tesi:SafePolynomial:arXiv2021}. All these formulations suffer the same limitations. \hfill $\square$
\end{remark}
\color{black}

To address the issue in \cref{rem:computationalcost}, in \cref{sec:relaxations} we propose two relaxations of \eqref{eq:main} that, at the price of some conservatism, result in substantial complexity reduction. However, first let us note that, if $\Pi$ is fixed, then \eqref{eq:main} reduces to an \gls{LMI}. This case is still relevant, as illustrated in the following two examples. 

\begin{example}[LTI systems] \label{ex:LTI}
Consider the system $\dot x =Ax+Bu$ (corresponding to \eqref{eq:system} for $N=1$), set $\Pi=0$ (justified by \cref{rem:nonpositive Metzlermatrix}), and note that \eqref{eq:LM1} recovers the standard Lyapunov inequality. Then, under \eqref{eq:Slaters}, \cref{th:main} proves that \eqref{eq:main} is \emph{necessary and sufficient} for quadratic stabilization of all the systems in $\C$. As an advantage with respect to \cite[Th.~14]{VanWaardeEtal_Noisy_TAC2022} (formulated in discrete-time), we use the nonstrict version of the S-lemma in our derivation, as \cref{lem:strictnonstrict} already bounds
\eqref{eq:LM1} away from zero over $\C$ at a solution --in fact, $Q\succ 0$ can be used to certify the worst-case (linear) convergence rate. \hfill $\square$. 
\end{example}

\begin{example}[Markov jump linear systems]\label{ex:Markov} Consider the system $\dot x = ( A_\sigma + B_\sigma K_\sigma) x $, where $\sigma(t)$ is an (uncontrolled, but measured) continuous-time Markov chain with known infinitesimal transition matrix $\Pi^\top$, $\Pi\in \M $. The system is stochastically stable (i.e., $ \int_0 ^\infty (E (\|x\|^2)) \text{d}t <\infty$ for any initial condition) if and only if \eqref{eq:LM1} admits a solution \cite[Th.~3.25]{Costa_MarkovJump_Springer2013}. In turn, the system of \glspl{LMI} \eqref{eq:main} provides \emph{nonconservative} conditions for stabilization if the system matrices are unknown, but open-loop experiments have been recorded. A numerical example is in \cref{sec:numerics:MJLS}.\hfill$\square$
\end{example} 

\begin{figure*}[!h]
\begingroup%
\thinmuskip=0mu plus 1mu
\medmuskip=0mu plus 2mu
\thickmuskip=1mu plus 3mu
\begin{align}
\label{eq:main2}
\tag{C2}
\begin{bmatrix}\\[-0.8em] 
       \tilde {\bs{Q}} & 0 & \tilde P_i & 0 & 0
        \\
        0& 
        \tilde {\bs{P}}_{\ms -i} & (\1 \otimes I)\tilde P_i & 0& 0 
        \\
        \Pit&   \Pit  (\1 \otimes I)^\top &  (N-1)\Pit  & -\tilde \gamma_i \Pit & - \bs L_i^\top  
        \\
       0 & 0&  - \tilde \gamma_i \Pit & 0 & 0 
       \\
       0& 0&  -\bs L_i & 0 & 0
\end{bmatrix}
    - \alpha_i \begin{bmatrix}
      0& 0 \\
      0 & 0 \\
      I & \dot X_i \\ 
      0 & - X_i \\ 
      0 & - U_i 
    \end{bmatrix}  
       \begin{bmatrix}
        \Phi_{1,1}^i & \Phi_{1,2}^i \\[0.5em]
        {\Phi_{1,2}^i}^\top & \Phi_{2,2}^i
    \end{bmatrix} 
       \begin{bmatrix}
      0& 0 \\
      0 & 0 \\
      I & \dot X_i \\ 
      0 & - X_i \\ 
      0 & - U_i 
    \end{bmatrix}^\top \succcurlyeq 0 
\end{align}
\vspace{-0.7em}
\hrule
\endgroup%
\end{figure*}

\section{Relaxations via structured weight matrix}\label{sec:relaxations}
%
%
In this section, we provide two  matrix inequalities, that are sufficient for the satisfaction of \eqref{eq:main}: the first 
is still a  \gls{BMI}, but of reduced dimension;
the second is more conservative, but only requires solving an \gls{LMI} with line-search on a scalar parameter. Both are obtained by assuming some extra structure on the \Metzler matrix $\Pi$.

\subsection{ \gls{BMI} of reduced order}
Let us impose,   for all $\i$, the additional constraints
\begin{align}\label{eq:specialPi_convex}
\pi_{j,i}= \gamma_i, \ \forall j\neq i; \quad \pi_{i,i} = -(N-1) \gamma_i,
\end{align}
where  $\gamma_i>0$ is a variable to be determined. 
To solve \eqref{eq:LM2} more efficiently under
\eqref{eq:specialPi_convex}, we replace the variable $Q\succ 0$ in \eqref{eq:LM2} with $\gamma_i \bs Q$, $\bs{Q} \succ 0$;
by dividing both sides by $\gamma_i$  we obtain: for all  $\i$, for all $(A_i,B_i)\in \C_i,$
 \begin{align}\label{eq:LM2modified}
\Acli^\top (\tilde \gamma_i P_i) +  (\tilde \gamma _i P_i)  \Acli + \textstyle \sum_{j\in \mc{I}\backslash\{i\}} (P_j-P_i)+ \bs{Q}\preccurlyeq 0.
\end{align}
%
%
By repeating the derivation in \cref{th:main},  we can show that, under the Slater's condition \eqref{eq:Slaters}, via the transformation $\bs{L}_i=\tilde \gamma_i K_i \tilde P_i $,  
\eqref{eq:LM2modified} is equivalent to \eqref{eq:main2}, on top of the next page.
\begin{theorem}[Relaxed data-driven stabilization]\label{th:relax1}
Assume that $(\{\tilde P_i \succ 0,\alpha_i\geq 0, \tilde \gamma_i >0, \bs L_i\}_{\i}, \bs Q\succ 0) $ satisfy \eqref{eq:main2}, for all $\i$. Then the controller in  \eqref{eq:control}, with $K_i = \gamma_i \bs L_i P_i$, globally asymptotically stabilizes \eqref{eq:system}. Furthermore, \eqref{eq:performance1} holds with $Q = \bs{Q} \left(\min_{\i}  \gamma_i \right)$. \hfill $\square$
\end{theorem}

\begin{remark}[Complexity reduction]
    The relaxation \eqref{eq:main2} is still a \gls{BMI}, but computationally much easier to solve than \eqref{eq:main}, where the  bilinearity involves significantly many more variables. 
For instance, \eqref{eq:main2} could be solved via \gls{SDP} and line-search over $N$ scalar variables -- instead of $N(N-1)$. 
As another example, let $M= \diag((M_i)_\i)$,
\begin{align}\label{eq:Mstructure}
    M_i\coloneqq \begin{bmatrix}
          1 &\tilde \gamma_i &  \operatorname{vec}(\tilde P_i)^\top
    \end{bmatrix}^\top \begin{bmatrix}
          1 & \tilde \gamma_i &  \operatorname{vec}( \tilde P_i)^\top
    \end{bmatrix},
\end{align}
and note that \eqref{eq:LM2modified} is an \gls{LMI} in the elements of $M$ and the variables $\{\bs L_i\}_\i$, $\tilde {\bs Q}$, $\{\alpha_i\}_{\i}$.
Therefore, the problem can  be recast as the following rank constrained  \gls{LMI}: find  $\tilde {\bs{Q}}$, $\{\bs L_i\}_\i$, $\{\alpha_i\}_{\i}$, $M$ such that:
\begin{align}\label{eq:rankconstrained}
\left\{
\begin{aligned}
&  M \succcurlyeq 0, \  \operatorname{rank}(M) = N,
\\
 &  \forall \i: \  [M_i]_{1,1} = 1, \ \eqref{eq:main2},  \tilde P_i \succ 0, \tilde{\gamma } _i > 0
   \end{aligned}
   \right.
\end{align}
where the constraints on $M$ enforce the structure in \eqref{eq:Mstructure}. This problem can be solved via recursive algorithms, where  each iteration requires solving an \gls{SDP} in all the variables (see \cite{Recht_SIAM2010} for an overview on the topic, or \cref{sec:numerics:switched:B} for an example). 
 In the case of \eqref{eq:main2}, the square matrix $M$ has $O(Nn^4)$ nonzero entries (variables); the corresponding number for the case of \eqref{eq:main}  is $O(N^3n^4)$. \hfill $\square$
\end{remark}

The rationale behind the relaxation in \eqref{eq:specialPi_convex} is that, in the known-model case, \cite[Eq.~39]{GeromelColaneri_SIAM2006} ensures that the \gls{LM} inequalities in \eqref{eq:LM1} admit a solution satisfying  \eqref{eq:specialPi_convex} if there is a Hurwitz convex combination of the matrices $\Acli$'s. Interestingly,
%
%
an analogous result holds for the data-driven case if the convex combination of the compatibility sets is \emph{quadratically} stabilizable. The proof of the following result is given in  Appendix~\ref{app:lem:existenceofspecialsol}. 
\begin{lemma}[A condition for solvability] \label{lem:existenceofspecialsol}
Assume that there exist  $\lambda\coloneqq(\lambda_i)_\i\in \Lambda_+$, $\{K_i\}_\i$,  $P \succ 0$ such that
\begin{align}\label{eq:quadratic}
\left( \textstyle \sum _\i \lambda_i \Acli \right  )   ^\top P + P  \left( \textstyle \sum _\i   \lambda_i \Acli \right) \prec 0
\end{align}
for all $(A_i,B_i)_\i \in \C$. Then, \eqref{eq:LM2} is solvable with $\Pi$  as in   \eqref{eq:specialPi_convex}, with $\gamma_i= \mu\lambda_i^{-1}$ and large enough $\mu>0$. Furthermore, if \eqref{eq:Slaters}  holds, then \eqref{eq:main2} is feasible. \hfill $\square$
\end{lemma}

The use of relaxations like \eqref{eq:specialPi_convex} is predominant in the recent data-driven results for nonlinear systems, e.g., \cite{Dai_Semialgebraic_LCSS2020,Guo:DePersis:Tesi:PolynomialNoisy:TAC2021}, allowing for tractable design.
While such relaxations are algebraically meaningful, a disadvantage is that it is usually not clear how restrictive they are, namely what conditions they impose  on the real system and on the data quality for the existence of a solution.
In contrast, \cref{lem:existenceofspecialsol} provides such an insight for \eqref{eq:specialPi_convex};  
this point is further illustrated in \cref{sec:Numerics}.
\color{black}

\begin{remark}[Comparison with classical conditions]
A well-known result from \cite{Liberzon2003} establishes that, if a solution to \eqref{eq:quadratic}  exists, then  the controller
\begin{align}\label{eq:classicinput}
 \sigma(x)=\argmin_\i \  x^\top ({\Aclibar{}}^\top P +  \Aclibar P) x, 
\end{align}
with $\Aclibar \coloneqq\bar A_i+ \bar B_iK_i$, stabilizes \eqref{eq:system}.
Yet, \eqref{eq:quadratic} is more restrictive than \eqref{eq:LM2modified} and,  most importantly, the controller \eqref{eq:classicinput} cannot be implemented in a data-driven fashion (in contrast to \eqref{eq:control}), because the true system matrices are not known.
\hfill $\square$
 \end{remark}



\subsection{LMI with line-search} 
Let us discuss a more conservative choice for the \Metzler matrix $\Pi$. For the structure in \eqref{eq:specialPi_convex},  we further assume the  parameters $\gamma_i$'s to be identical, so that
\begin{align}
\label{eq:specialPi_average}
 \ \Pi = \gamma (\textstyle \1_N \1_N ^\top-N I_N),
\end{align}
for some scalar variable $\gamma> 0$. 
The advantage is computational: for $\gamma$ fixed, \eqref{eq:main} (or \eqref{eq:main2}) reduces to an \gls{LMI}. Thus, a solution can be sought via standard convex solvers and line-search over \emph{one} scalar variable $\gamma$. 

\cref{lem:existenceofspecialsol} provides a sufficient condition for existence of a solution to \eqref{eq:main2} under \eqref{eq:specialPi_average}, i.e., quadratic stabilizability of the ``average'' compatibility set.
If a solution is found, then \cref{th:relax1} provides a stabilizing controller. 

\begin{remark}[Comparison with a known relaxation]\label{rem:less conservative} In \cite[Th.~4]{GeromelColaneri_SIAM2006} (see also \cite{Geromel:Colaneri:Bolzern:TAC:2008,Geromel:Deacto:Daafouz:Suboptimal:TAC2013}) the following relaxation to the (model-based) \gls{LM} inequalities \eqref{eq:LM1} is proposed:
\begin{align}\label{eq:LM_modified}
\! \Acli^{\!\top} P_i + P_i \Acli+\gamma(P_j - P_i) +Q \prec 0, \ \forall j\neq \i, \!
 \end{align}
with $\gamma>0$. These conditions are also \glspl{LMI} for fixed $\gamma$.
By multiplying both sides of \eqref{eq:LM_modified} by  $\pi_{j,i}/ \gamma$ and summing over 
$j$, we see that any solution  $(\{P_i,K_i\}_\i,Q,\gamma)$ of \eqref{eq:LM_modified} also solves \eqref{eq:LM1} with $\Pi$ as in \eqref{eq:specialPi_average}. 
Hence, our proposed relaxation \eqref{eq:LM1}-\eqref{eq:specialPi_average} is (strictly, see an example in  \cref{sec:numerics:switched:A}) less restrictive than \eqref{eq:LM_modified} --as well as computationally convenient: the number of \glspl{LMI} in \eqref{eq:LM_modified} is quadratic in $N$. \hfill $\square$
\end{remark}




\section{Performance specifications}\label{sec:Extensions:H}
\begin{figure*}[!h]
\begingroup%
\thinmuskip=0mu plus 1mu
\medmuskip=0mu plus 2mu
\thickmuskip=1mu plus 3mu
\begin{align}
\label{eq:mainH2}
\tag{$\H_2$}
\begin{bmatrix}\\[-0.8em]
        \tilde{\bs{\pi}}_{-i}\tilde {\bs{P}}_{\ms -i}  & (\1 \otimes I) \tilde P_i & 0 & 0& 
        \\
        \tilde P_i (\1 \otimes I_n)^\top & -\pi_{i,i}\Pit -Y_i & -\Pit & -L_i^\top  \\
        0&  -\Pit & 0 & 0 \\
         0&  -L_i & 0 & 0
    \end{bmatrix} - \alpha_i \begin{bmatrix}
      0 & 0 \\
      I & \dot X_i \\ 
      0 & - X_i \\ 
      0 & - U_i 
    \end{bmatrix}  
        \Phi^i 
       \begin{bmatrix}
      0 & 0 \\
      I & \dot X_i \\ 
      0 & - X_i \\ 
      0 & - U_i 
    \end{bmatrix}^\top  \succcurlyeq 0, \quad  \begin{bmatrix}
          I &  C_i\tilde P_i +D_i L_i \\
          \star  & Y_i
    \end{bmatrix}\succ 0
\end{align}
\endgroup%
\end{figure*}
\begin{figure*}[!h]
\vspace{-1.5em}
\begingroup%
\thinmuskip=0mu plus 1mu
\medmuskip=0mu plus 2mu
\thickmuskip=1mu plus 3mu
\begin{align}
\label{eq:mainHinf}
\tag{$\H_\infty$}
\begin{bmatrix}\\[-0.8em]
        \tilde{\bs{\pi}}_{-i}\tilde {\bs{P}}_{\ms -i}  & (\1 \otimes I) \tilde P_i & 0 & 0& 
        \\
        \tilde P_i (\1 \otimes I_n)^\top & -\pi_{i,i}\Pit -Y_i & -\Pit & -L_i^\top  \\
        0&  -\Pit & 0 & 0 \\
         0&  -L_i & 0 & 0
    \end{bmatrix} - \alpha_i \begin{bmatrix}
      0 & 0 \\
      I & \dot X_i \\ 
      0 & - X_i \\ 
      0 & - U_i 
    \end{bmatrix}  
        \Phi^i 
       \begin{bmatrix}
      0 & 0 \\
      I & \dot X_i \\ 
      0 & - X_i \\ 
      0 & - U_i 
    \end{bmatrix}^\top \succcurlyeq 0, \quad  \begin{bmatrix}
          Y_i & E_i &   (C_i\tilde P_i +D_i L_i)^\top \\
          \star  & \rho I & -F_i^\top \\
           \star & \star & I
    \end{bmatrix}\succ 0 \end{align}
\hrule
\endgroup%
\end{figure*}

In this section we study quantitative
performance specifications. We consider a model like~\eqref{eq:system} with disturbances in the dynamics and outputs,
\begin{subequations}\label{eq:system2}
    \begin{align}
    \label{eq:system2:a}
    \dot{x} &  =\bar{A}_\sigma x+ \bar B_\sigma u + E_\sigma \dist \\
    \label{eq:system2:b}
    z & = C_\sigma x + D_\sigma u + F_\sigma \dist,
\end{align}
\end{subequations}
where $\dist(t)\in\R^q$ is an exogenous disturbance and $z \in \R^p$ is a performance output;  $\{C_i,D_i, E_i,F_i\}_\i$ are \emph{known} matrices. The matrices  $\{C_i,D_i,F_i\}_\i$ are chosen by the designer to define the performance objective. The matrices 
$\{E_i\}_\i$ measure the influence of the disturbance signal on the state evolution, and can be used to encode structural prior knowledge about the disturbance (e.g., if $\psi$ only acts on some state components; if no prior information is available, then $E_i=I$ is chosen).
%
%
Further, the only information available on the matrices $\{\bar A_i,\bar B_i\}_\i$ is a set of data satisfying \eqref{eq:data}.
%
%

We next address the  $\H_2$ and $\H_\infty$ stabilization problems for \eqref{eq:system2}. We refer to  Appendix~\ref{app:LM:performance} for a formal description of the performance indices $J_2$ and $J_\infty$: in analogy to the  \gls{LTI} case, these can be defined in terms of the system response to impulsive and integrable disturbances, respectively. 
The problems are well known to be hard to solve for switched linear systems, even when a model is perfectly known, and thus tight formulation are usually replaced by sufficient conditions \cite{Geromel:Colaneri:Bolzern:TAC:2008,Antunes:Heemels:LQRswitched:TAC2017,Deaecto:Geromel:Hinf:ASME:2010}; in particular, in this paper we rely on the \gls{LM} inequalities reviewed in \cref{app:LM:performance}.
For the continuous input, we again only focus on linear switched controllers $u = K_\sigma x $. 

\begin{remark}[Refined disturbance model]\label{rem:refineddisturbance}
Assume that the data \eqref{eq:data} are generated by \eqref{eq:system2:a}, where the quantities $\bar \noiseset_i$ in \eqref{eq:data} are due to the unknown process disturbance $\dist$ in \eqref{eq:system2:a}: in short, $\noise(t^i_k) = E_i \dist(t^i_k)$. Then, if $\bar{\distset}_i \coloneqq  \begin{bmatrix} \dist(t_1^ i) & \dist(t_2^ i) & \dots & \dist(t_{T_i}^ i) \end{bmatrix}$ satisfies 
\begin{align*}
\begin{bmatrix}
    I \\  \bar \distset_i^\top 
    \end{bmatrix}^\top  \begin{bmatrix}
        \hat \Phi_{1,1}^i & \hat \Phi_{1,2}^i \\[0.5em]
        {\hat{\Phi}_{1,2}}^{i^\top} & \hat \Phi_{2,2}^i
    \end{bmatrix} 
       \begin{bmatrix}
    I \\  \bar \distset_i^\top 
    \end{bmatrix} \succcurlyeq 0, \end{align*} with $\hat \Phi^i_{2,2} \prec 0$ (e.g., an energy-type prior on the process noise $\dist$), then \cref{asm:disturbancemodel} is satisfied with  $\Phi^{i}_{1,1} = E_i\hat \Phi^i_{1,1}E_i^\top$, $\Phi_{1,2}^i = E_i \hat \Phi^{i}_{1,1} $, $ \Phi_{2,2}^i = \hat \Phi_{2,2}^i$ (see  \cite[Rem.~2]{VanWaardeEtal_Noisy_TAC2022}), which allows us to incorporate the prior knowledge on the matrix $E_i$. 
    Additional noise/estimate error  can also be taken into account by properly choosing  $\Phi_{i}$. 
    \hfill $\square$
\color{black}
\end{remark}

\subsection{Data-driven $\mc{H}_2$ control}\label{sec:Extensions:H2} 

Following \cref{prop:H2} in \cref{app:LM}, to design a controller with guaranteed $\H_2$ performance, we impose: for all  $\i$, for all $(A_i,B_i)\in \C_i$,
\begin{align}\label{eq:LMH2} 
\Acli ^\top P_i + P_i \Acli + \textstyle \sum_{j\in \mc{I}}  \pi_{j,i} P_j+ \Ccli^\top \Ccli \prec  0,
\end{align}
where $\Ccli \coloneqq C_i +D_iK_i$. 
 Ideally, we should introduce no additional nonlinearity with respect to \eqref{eq:main} (we recall that $K_i$ is a variable, and further the matrix $\Ccli^\top \Ccli$ might be singular, thus not positive definite). This goal is achieved in the  following result. 

\begin{theorem}[Data-driven $\H_2$ stabilization]\label{th:mainH2}
Let  $(\{ P_i\succ 0,\alpha_i \geq 0,L_i,Y_i \succ 0\}_\i, \Pi\in \Mp)$ satisfy \eqref{eq:mainH2}, where the matrix $\Phi_i$ is given in \cref{asm:disturbancemodel}, for all $\i$.
%
%
Then the controller in \eqref{eq:control} with $K_i=L_i P_i$ globally asymptotically stabilizes the switched system in \eqref{eq:system2}. Furthermore, if $E_i =E$ for all $\i$, then  the closed-loop system satisfies $
    J_2 (\sigma,u) < \min \  \{\operatorname{tr} (E^\top P_i E ), i\in \mc{I} \}. 
   $
\hfill $\square$
\end{theorem}

\begin{figure*}[!h]
\begingroup%
\thinmuskip=0mu plus 1mu
\medmuskip=0mu plus 2mu
\thickmuskip=1mu plus 3mu
\begin{align}
\label{eq:maininvariance}
\tag{C4}
\begin{bmatrix}\\[-0.8em]
        \beta_j I & E^\top & 0 & 0
        \\
        E &
        (\pi_{j,j}-\beta_j)\Pjt +  \Pjt \left(\sum_{k \neq j }   \pi_{k,j}P_k \right) \Pjt & -\Pjt & -L_j^\top  
        \\
      0&  -\Pjt & 0 & 0 
        \\
        0&   -L_j & 0 & 0
    \end{bmatrix} - \alpha_j \begin{bmatrix}
      0 & 0 \\
      I & \dot X \\ 
      0 & - X \\ 
      0 & - U
    \end{bmatrix}  
              \begin{bmatrix}
        \Phi_{1,1} & \Phi_{1,2}\\[0.5em]
        {\Phi_{1,2}}^\top & \Phi_{2,2}
    \end{bmatrix} 
      \begin{bmatrix}
      0 & 0 \\ 
      I & \dot X \\ 
      0 & - X \\ 
      0 & - U 
    \end{bmatrix} ^\top\succcurlyeq 0
\end{align}
  \vspace{-2em}
\begin{align}
\label{eq:maininvariance2}
\tag{C5}
\begin{bmatrix}\\[-0.8em]
        \bs \beta_j I & E^\top & 0 & 0 
        \\
        E & -(M+\bs \beta_j)\Pjt + \Pjt \left(  \sum_{k\neq j} P_k \right) \Pjt  & -\tilde \gamma_j\Pjt & -\bs L_j^\top  \\
        0 & -\tilde \gamma_j\Pjt & 0 & 0 
        \\
        0 &   -\bs L_j & 0 & 0
    \end{bmatrix} - \bs \alpha_j \begin{bmatrix}
      0 & 0 \\
      I & \dot X \\ 
      0 & - X \\ 
      0 & - U
    \end{bmatrix}  
          \begin{bmatrix}
        \Phi_{1,1} & \Phi_{1,2}\\[0.5em]
        {\Phi_{1,2}}^\top & \Phi_{2,2}
    \end{bmatrix} 
      \begin{bmatrix}
      0 & 0 \\
      I & \dot X \\ 
      0 & - X \\ 
      0 & - U 
    \end{bmatrix}^\top  \succcurlyeq 0
 \end{align}
\hrule
\endgroup%
\end{figure*}

\begin{pf}
By taking the Schur complement with respect to the upper-left block and by \cref{lem:S}, we deduce that \eqref{eq:mainH2} implies:
for all $\i$, for all $(A_i,B_i)\in \C_i$,
\begin{subequations}\label{eq:useful3}
\begin{align}\label{eq:useful3:a}
- \tilde P_i \Acli ^\top  -  \Acli \tilde P_i - \textstyle \sum_{j\in \mc{I}}  \pi_{j,i} \tilde P_i  P_j\tilde P_i -  Y_i \succcurlyeq  0,
\\\label{eq:useful3:b}
Y_i \succ (C_i \tilde P_i +D_i L_i)^\top (C_i \tilde P_i +D_i L_i)
 \end{align}
\end{subequations}
where  $K_i \coloneqq L_i P_i$. 
 Multiplying both sides of \eqref{eq:useful3:a} by $P_i$ and using \eqref{eq:useful3:b}, we retrieve \eqref{eq:LMH2}, because $C_i\Pit +D_i L_i = \Ccli \Pit$. The conclusion follows by \cref{prop:H2}. 
\hfill $\blacksquare$
\end{pf}

\subsection{Data-driven $\mc{H}_\infty$ control}\label{sec:Extensions:Hinfty} 

The following is the data-driven counterpart of \cref{prop:Hinf} in Appendix~\ref{app:LM}.

\begin{theorem}[Data-driven $\H_\infty$ stabilization]\label{th:main:Hinf}
Let $(\{ P_i\succ 0,\alpha_i \geq 0\,L_i ,Y_i \succ 0\}_\i, \Pi\in \Mp, \rho>0 )$ satisfy \eqref{eq:mainHinf}, where the matrix $\Phi_i$ is given in \cref{asm:disturbancemodel}, for all $\i$. Then the controller in \eqref{eq:control} with $K_i=L_i P_i$ globally asymptotically stabilizes the switched system in \eqref{eq:system2}. Furthermore,  the closed-loop systems satisfies    $J_\infty (\sigma,u) < \rho.$ \hfill $\square$ 
\end{theorem}
\begin{pf}
 By applying the Schur complement with respect to the two bottom-right blocks, the second inequality of \eqref{eq:mainHinf} is equivalent to
 \begin{align*} Y_i & \succ  \begin{bmatrix} E_i^\top \\
    \Ccli \tilde P_i
    \end{bmatrix}^\top
    \begin{bmatrix}
          \rho_I  & -F_i^\top \\
          \star & I 
    \end{bmatrix}^{-1}
     \begin{bmatrix} E_i^\top \\
    \Ccli \tilde P_i
    \end{bmatrix}
 \end{align*}
 where we used the definition $C_i \tilde P_i +D_i L_i = \Ccli \tilde P_i $. By replacing the previous inequality in \eqref{eq:useful3:a}, by multiplying both sides by $P_i$, and by a Schur complement argument, we obtain: for all  $\i$, for all $(A_i,B_i)\in \C_i$,
 \begin{align*}
     \begin{bmatrix}
           \Acli^\top P_i +P_i\Acli + 
          \textstyle \sum_{j\in \mc{I}} \pi_{i,j} P_j & P_i E_i & \Ccli^\top
          \\
          \star & -\rho I & F_i^\top  \\
          \star & \star & -I
     \end{bmatrix} \prec 0. 
 \end{align*}
 The conclusion follows by \cref{prop:Hinf}. \hfill $\blacksquare$
\end{pf}

The relaxations in \cref{sec:relaxations} can also be applied to reduce the computational cost in \eqref{eq:mainH2} and \eqref{eq:mainHinf}. We note that, for a \emph{fixed} matrix $\Pi$, optimizing the cost $\min_{{\i}} \operatorname{tr} (E^\top P_i E)$ over the solutions  of \eqref{eq:mainH2} reduces to $N$ \gls{SDP} problems. 
Instead the upper bound $\rho$ in \cref{th:main:Hinf} is conveniently  linear, thus its minimization corresponds to only one \gls{SDP}.


\redone{

\section{Discrete-time switched linear systems}\label{sec:DiscreteTime}
\begin{figure*}[!h]
\begingroup%
\thinmuskip=0mu plus 1mu
\medmuskip=0mu plus 2mu
\thickmuskip=1mu plus 3mu
\begin{align}
\label{eq:main_DT}
\!
\begin{bmatrix}\\[-0.8em]
        \tilde{\bs{\Lambda}}_{-i}\tilde {\bs{P}}_{\ms -i} + (\1 \otimes I)\tilde \lambda_{i,i} \tilde P_i (\1 \otimes I)^\top  & (\1 \otimes I) \tilde \lambda_{i,i} \tilde P_i & 0 & 0& 0
        \\
        \tilde \lambda_{i,i} \tilde P_i (\1 \otimes I_n)^\top & \tilde \lambda_{i,i}\Pit & 0 & 0   & 0
        \\
        0 & 0 & -Y_i & - L_i ^\top & 0
        \\
         0 & 0 &  -L_i & 0 & L_i 
         \\
         0 & 0 & 0 & L_i^\top &  Y_i
    \end{bmatrix}\! - \alpha_i \begin{bmatrix}
      0 & 0 \\
      I & \dot X_i \\ 
      0 & - X_i \\ 
      0 & - U_i \\
      0 & 0
    \end{bmatrix}  
        \Phi^i 
       \begin{bmatrix}
      0 & 0 \\
      I & \dot X_i \\ 
      0 & - X_i \\ 
      0 & - U_i \\
      0 & 0 
    \end{bmatrix}^{\hspace{-0.3em} \top} \hspace{-0.3em} \succcurlyeq 0, 
    \quad  
    \begin{aligned}
    \tilde Q - \Pit \succ 0
    \\
    \begin{bmatrix}
          Y_i - \tilde P_i &  \tilde P_i
          \\
          \tilde P_i &   \tilde Q - \Pit
    \end{bmatrix} \! \succcurlyeq 0
    \end{aligned}
    \tag{C3}
\end{align}
\hrule
\endgroup%
\end{figure*}

Data-driven stabilization of \emph{discrete-time} switched linear systems via \gls{LM} inequalities presents
some additional technical difficulties, which we address in this section. We consider the (nominal) plant
\begin{align}\label{eq:DTsystem}
    x(k+1) =  \bar A_{\sigma (k)} x(k) + \bar B_\sigma u(k),
\end{align}
where $\sigma \in \I$ is the switching signal, and $\{\bar {A}_i \in \R^{n\times n},\bar B_i \in \R^{n \times m} \}_{\i} $ are the unknown true system matrices. Analogously to the continuous-time case, we assume that for each mode $\i$, $T_i> 0$ samples have been collected for each subsystem at sampling times $\{t^i_s\}_{s=1,2,\dots,T_i} \subset \mathbb{N}$, satisfying 
\begin{align}\label{eq:DTdata}
    X_i^+ = \bar A_i X_i + \bar B_i U_i + \bar \noiseset_i,
\end{align}
where 
\begin{align*}
     X_i^{+} & \coloneqq  \begin{bmatrix}  x(t_1^ i +1)  &  x (t_2^ i+1) & \dots &  x(t_{T_i}^ i+1)  \end{bmatrix},
\end{align*}
and the other quantities in \eqref{eq:DTsystem} are defined as in \eqref{eq:data}, with $\noise$ being an unknown process noise satisfying \cref{asm:disturbancemodel}. Note that, for each mode, this setup retrieves exactly that considered for \gls{LTI} systems, e.g., in \cite{VanWaardeEtal_Noisy_TAC2022,DePersis_Formulas_TAC2020}. 

Our first result provides a data-driven version of the  \gls{LM} inequalities in \cite[Lem.~2]{GeromelColaneri_IJC2006}, which read:  for all $\i$,
\begin{align}\label{eq:DTLM}
\Acli^\top \left( \textstyle  \sum _{j\in \I} \lambda_{i,j} P_j
\right ) \Acli - P_i + Q \prec 0 ,
\end{align}
where $\{P_i \succ 0 \}_{\i}$, $\Lambda = [\lambda_{i,j} ]_{i,j \in \I} \in \mc{S}$, $Q \succcurlyeq 0$ and $\Acli \coloneqq A_i +B_i K_i$. We wish to impose these conditions for all the matrices compatible with the experiments. Similarly to \cref{lem:strictnonstrict}, we can consider without loss of generality 
the variant:  for all $ \i$, for all $(A_i,B_i) \in \C_i$,
\begin{align}\label{eq:DTLM:b}
\Acli^\top \left( \textstyle \sum _{j\in \I} \lambda_{i,j} P_j
\right ) \Acli - P_i + Q \preccurlyeq 0,
\end{align}
with $\{P_i \succ 0 \}_{\i}$, $Q \succ 0$ and $\Lambda \in \mc{S}_+$, $P_i -Q \succ 0$.
%
%

Note that in \eqref{eq:DTLM:b}, the quantity $\Acli^\top$ is a factor on the left, which prevents us from directly applying the S-lemma. A standard dualization argument (i.e., double application of the Schur complement) gives the equivalent inequality 
\begin{align}\label{eq:DTdual}
 \left(\textstyle \sum_{j\in\mc{I}} \lambda_{j,i} P_j \right)^{-1} - \Acli  (P_i -Q)^{-1} \Acli^\top \succcurlyeq 0.
 \end{align}
Differently from the continuous-time case,   the first term of \eqref{eq:DTdual} is the \emph{inverse of the sum} of variables, a nonlinearity that cannot be dealt with by using the Schur complement.  
To cope with this complication, we leverage  the Woodbury identity, which we hereby recall:
\begin{align}\label{eq:Woodbury}
(D+UCV)^{-1}=\tilde D-\tilde DU(\tilde C +V\tilde D U)^{-1}V \tilde D,
\end{align}
for any suitably invertible and sized matrices $C,D,U,V$.    
The following result provides a \gls{DT} counterpart of \cref{th:main} and it is the main result of this section. 

\begin{theorem}[\gls{DT} data-driven LM inequalities]\label{th:DTmain}
Assume that the Slater's condition in \eqref{eq:Slaters} holds, and let  $\Pbs \coloneqq \diag ((P_i)_{j\in \I \backslash \{i\}})$, $\bs{\Lambda}_{-i} \coloneqq \diag ((\lambda_{j,i}I_n)_{j\in \I \backslash \{i\}}) $, for all $\i$. Then, $(\{P_i\succ 0\}_{\i},\{K_i\}_\i,\Lambda \in \mc{S}_{+},Q\succ 0)$ solve \eqref{eq:DTdual} if and only if there exist $\{\alpha_i\geq 0,Y_i \succ 0\}_{\i}$ such that
$(\{P_i,Y_i,L_i\coloneqq K_i Y_i,\alpha_i \}_\i,\Lambda,Q)$ verify the inequality \eqref{eq:main_DT}, where  $\Phi_i$ is as in \cref{asm:disturbancemodel}, for all $\i$. \hfill $\square$
\end{theorem}  
\begin{proof}
By applying the Schur complement with respect to the top-left block in the first inequality of \eqref{eq:main_DT}, we get a (4-by-4 block) inequality, where the component of the top-left block derived by the first matrix (i.e., excluding the matrix multiplied by $\alpha_i$ in \eqref{eq:main_DT}) is, with $\ast = \tilde \lambda_{i,i} \Pit$ for brevity,
\begin{align*}
    & \ast -\ast (\1 \otimes I)^\top \left( \tilde{\bs{\Lambda}}_{-i}\tilde {\bs{P}}_{\ms -i} + (\1 \otimes I)\ast(\1 \otimes I)^\top \right)^{-1}(\1 \otimes I)\ast
    \\
    & = ( \lambda_{i,i}  P_i +(\1 \otimes I)^\top {\bs{\pi}}_{-i} {\bs{P}}_{\ms -i} (\1 \otimes I) )^{-1}
    \\ 
    & = (\textstyle \sum_{j\in \I} \lambda_{j,i} P_j)^{-1} 
\end{align*}
where the first equality is due to \eqref{eq:Woodbury}. Subsequently, by applying the Schur complement to the bottom right block we get a (3-by-3 block) inequality where the component of the bottom-right block derived by the first matrix is $ - L_i \tilde Y_i  L_i ^\top$. Then, by \cref{lem:S}, the first inequality of \eqref{eq:main_DT} is equivalent to: for all $(A_i,B_i)\in\C_i$
\begin{align}\label{eq:DTdualform}
    \begin{bmatrix}
    I \\ A_i ^\top \\ B_i^\top 
    \end{bmatrix}^\top \ms 
     \ms  \ms \begin{bmatrix}
        (\sum_{j\in \I} \lambda_{j,i} P_j)^{-1} & 0  & 0
        \\
        0 & -Y_i & -L_i^\top \\
        0 & -L_i ^\top  & - L_i \tilde Y_i L_i^\top
    \end{bmatrix} \ms \ms \ms 
  \begin{bmatrix}
    I \\ A_i ^\top \\ B_i^\top 
    \end{bmatrix} \ms \succcurlyeq 0.
\end{align}
The second inequality in \eqref{eq:main_DT} forces $P_i - Q \succ 0$, while the third, by  Schur complement,  is
\begin{align}
    Y_i & \succcurlyeq \Pit - \Pit (\Pit -\tilde Q)^{-1} \Pit = (P_i -Q)^{-1},
\end{align}
where the equality is again \eqref{eq:Woodbury}. Together with \eqref{eq:DTdualform} and by definition of $L_i$,  this implies
\eqref{eq:DTdual}; conversely, if  \eqref{eq:DTdual} holds, there exists $\alpha_i$ that satisfies \eqref{eq:main_DT}   with $Y_i = (P_i-Q)^{-1}$.
\end{proof}

As an immediate consequence of \cref{th:DTmain} and \cite[Lem.~2]{GeromelColaneri_IJC2006}, we have the following result.

 \begin{corollary}[\gls{DT} stabilization] \label{cor:DTmain}
Assume that \sloppy $(\{  P_i\succ 0,\alpha_i \geq 0,L_i,Y_i \succ 0,  \}_\i, Q\succ  0,\Lambda\in\mc{S}_+)$ solve \eqref{eq:main_DT}  for all $\i$.  Then the controller \eqref{eq:control} with $K_i= L_i \tilde Y_i$ globally asymptotically stabilizes the discrete-time switched system in \eqref{eq:DTsystem}. Furthermore, it holds that 
\begin{equation*}
  \textstyle  \sum_{k=0} ^\infty x(k)^\top Q x(k) \leq \min_{\i} x(0)^\top P_i x(0).   \QEDopenhereeqn
\end{equation*}
\end{corollary}



Casting \eqref{eq:main_DT} as a \gls{BMI} requires some auxiliary variables and constraints to impose $\Lambda \in \mc{S}_+$, analogously to \eqref{eq:main}-\eqref{eq:rankconstrained}. Nonetheless, as in \eqref{eq:main2}, this complication can be avoided  when employing a relaxation analogous to \eqref{eq:specialPi_convex} (i.e., imposing $\lambda_{j,i} = \gamma_i$, for all $j\neq i$ and some $0<\gamma_i<\frac{1}{N}$). Moreover, \eqref{eq:main_DT} is an \gls{LMI} when $\Lambda$ is fixed; this is relevant, for instance  because existence of a solution to \eqref{eq:DTLM} is \emph{equivalent} to stability of a \gls{DT} Markov jump linear system with transition matrix $\Lambda^\top$ (and system matrices $\Acli$), cf. \cref{ex:Markov}.

Finally, similar techniques to \cref{th:DTmain} can  be leveraged to address data-driven $\H_2$ and $\H_\infty$ problems for \gls{DT} switched linear systems, for example based on the results in \cite{Fioravanti:Goncalves:Deaecto:Geromel:Alternative:Franklin:2013}.


}

\section{Switched data-driven compensators for robust constrained stabilization of \gls{LTI} systems}\label{sec:Constrained}


%
%
In this section, we build upon the design ideas, developed above for switched systems, to design a switched controller for a \emph{fixed} linear plant. 
In particular,
we study a  robust constrained control problem, motivated by applications where it is paramount to keep a plant in safe operating conditions, despite the presence of disturbances. 
Let us consider the perturbed \gls{LTI} system
\begin{align}\label{eq:disturbedLTI}
\dot{x} = \bar Ax + \bar Bu +E \dist,
\end{align}
$x\in \R^n$, $u\in\R^m$, where the (persistent) disturbance $\dist(t) \in \R^q$ satisfies $\dist(t)^\top \dist(t) \leq 1$ for all times, subject to  polyhedral  state constraints described by  $x \in \mc{X}$,
\begin{align}
    \mc{X} \coloneqq \left\{ x \mid |c_j ^\top x| \leq 1,  \forall j \in \mc{J} \coloneqq \{1,2,\dots,M\} \right\},
\end{align}
for some $\{c_j\in \R^n \}_{j \in \mc{J}}.$
%
%
One fundamental challenge is to find an (as large as possible) set $\mc{X}_0 \subseteq \mc{X}$, together with a feedback controller that makes the set $\mc{X}_0$ invariant for the closed-loop dynamics \eqref{eq:disturbedLTI} --for any possible $\dist$-- and ensures asymptotic stability for the nominal system in the absence of disturbance. 
Here we draw from \cite{Thibodeau:Tong:Hu:SetInvariance:AUT:2009}, that leverages the Lyapunov function $v(x) = \max\{ x^\top P_0x, x^\top c_1 c_1^\top x,\dots,x^\top c_N c_N^\top x \} $ (for some $P_0 \succ 0$) for the estimation of maximal invariant sets. 
We depart from  \cite{Thibodeau:Tong:Hu:SetInvariance:AUT:2009} by designing a novel switched controller, that further ensures asymptotic stability, and without requiring identification of the system in \eqref{eq:disturbedLTI}.

In particular, we assume that the matrix $E\in \R^{n\times q}$ is known, but the only information available on the matrices $(\bar A, \bar B)$ is a set of data satisfying \eqref{eq:data}\footnote{Since $N$ =1, we omit the subscript $i=1$ in this section, see also \eqref{eq:maininvariance}.}. For each $j \in \mc{J}$, we define
\begin{align}\label{eq:Pj}
P_j \coloneqq c_jc_j^\top + \nu I  \succ 0,
\end{align}
where $\nu >0 $ is a fixed design regularization constant, and
\begin{align}\label{eq:Vmax}
    v_{\max} (x) \coloneqq \max \left \{ x^\top P_j x \mid    j \in \mc{J}_0\coloneqq \mc{J} \cup \{0\} \right\},
\end{align}
where $P_0 \succ 0$ is to be designed. Furthermore let $\mc{X}_0$ be the $1$-sublevel set of $v_{\max}$, i.e., 
\begin{align}\label{eq:X0}
    \mc{X}_0 \coloneqq \{ x \in \R^n \mid v_{\max} (x) \leq 1 \} \subset \mc{X}.
\end{align}
\begin{theorem}[Data-driven safe stabilization] \label{th:invariance}
Let $(\{P_j \succ 0 , L_j, \beta_j \geq 0, \alpha_j \geq 0\}_{j\in \mc{J}_0 }, \Pi = [ \pi_{j,k}]_{j,k \in \mc{J}_0} \in \Mp)$ satisfy \eqref{eq:maininvariance}, for all $j \in \mc{J}_0$.
Consider the controller 
\begin{align}\label{eq:inputswitchmax}
    \sigma(x) & =  \min \{ \operatorname{argmax}_{\jzero} x^\top P_j x\}, \quad 
    u(x) = K_{\sigma(x)} x,
\end{align}
where $K_j \coloneqq L_jP_j$. Then, any Carathéodory solution $x:\R_{\geq 0} \rightarrow \R^n$ of the closed-loop system \eqref{eq:disturbedLTI}, \eqref{eq:inputswitchmax}, with  $x(0) \in \mc{X}_0 $, satisfies $x(t) \in \mc{X}_0$ for all $t\geq 0$. Further, if $\beta_j >0$ for all $\jzero$ and $\dist(t) \rightarrow 0$ for $t \rightarrow \infty$, then also $x(t) \rightarrow 0$. \hfill $\square$ 
\end{theorem}
\begin{pf}
By taking the Schur complement with respect to the top-left block and by \cref{lem:S}, we deduce that \eqref{eq:maininvariance} implies: for all  $  \jzero$, for all $ (A,B)\in \C,$
\begin{align*}
-\tilde P_j \Aclj^\top -  \Aclj \tilde P_j - \beta_j\Pjt - \tilde\beta_j EE^\top+  \textstyle \underset{k \in \mc{J}_0}{\textstyle \sum} \pi_{k,j} \Pjt P_k \Pjt \succcurlyeq  0,
\end{align*}
where $\Aclj \coloneqq A+BK_j$. Multiplying both sides by $P_j$, we obtain via  Schur complement that
\begin{align}\label{eq:useful_invariance}
    \begin{bmatrix}
          \Aclj ^\top P_j+ P_j \Aclj + \beta_j P_j - \underset{k\in \J_0 }{\textstyle \sum } \pi_{j,k}P_k  & \star \\[1.5em]
          E^\top P_j & -\beta_j I
    \end{bmatrix} \preccurlyeq 0
\end{align}
    for all $\jzero$.
 We recall that the directional derivative of $ v_{\max}$ along $\zeta$ is $\dot v_{\max} (x,\zeta) = \max \{ 2x^\top  P_j \zeta \mid j \in \mathrm{I}_{\max} (x)\}$, where $I_{\max} (x) = \argmax_{ \jzero} x^\top P_j x $ \cite[Eq.~5]{Thibodeau:Tong:Hu:SetInvariance:AUT:2009}. By multiplying \eqref{eq:useful_invariance} on the left by any $ [ 
     x^\top \ \  \dist^\top ]$ and on the right by its transpose, we have for some $j \in \mathrm{I}_{\max} (x) $ that
    \begin{align}
    \nonumber
        & \hphantom{{}\leq{}}2\dot v_{\max} (x,\Aclj x+E\dist) = 
     2x^\top  P_j (\Aclj x+E\dist) 
     \\ 
     \label{eq:invariance_usstep_2}
     & \leq  {\textstyle \sum_{k\in \J_0} } \pi_{j,k} \ x^\top P_k x + \beta_j (\dist ^\top \dist  - x^\top P_j x ) \end{align}
     Note that ${\textstyle \sum_{k\in \J_0 }} \pi_{j,k} \ x^\top P_k x \leq 0$ (because $j \in \mathrm{I}_{\max} (x)$ and $\Pi \in \M)$. The conclusion follows because the second addend in  \eqref{eq:invariance_usstep_2} is: nonpositive, for each $x$ on the boundary of $\mc{X}_0$ (as $x^\top P_j x  = 1$ and $\dist^\top \dist \leq 1)$; negative, for any  $x\neq 0$ and small enough disturbance $\dist$, if $\beta_j >0$.  \hfill $\blacksquare$
\end{pf}
If we further impose the structure in \eqref{eq:specialPi_convex} for the \Metzler matrix $\Pi$, \eqref{eq:maininvariance} simplifies as in \eqref{eq:maininvariance2} (via the changes of variable $\bs{\beta}_j = \tilde \gamma_j \beta_j >0$, $ \bs{L_j}=\tilde \gamma_j L_j$, $\bs \alpha_j =\tilde \gamma_j \alpha_j >0$), which is a \gls{BMI}. Besides being easier to solve (see \cref{sec:numerics:switched}), an advantage of \eqref{eq:maininvariance2} is that the number of bilinear terms does not depend on the number of constraints $M$ (note that the matrices $\{ P_j \}_{j\in \mc{J}}$ are fixed  a priori   --they are not variables).

\begin{remark}[Maximizing $\mc{X}_0$] \label{rem:maxX0}
To maximize the volume of the invariant set $\mc{X}_0$, we can fit inside $\mc{X}_0$ an ellipsoid $\mc{X}_Q\coloneqq \{x\in \R^n \mid x^ \top Q x \leq 1\}$ of maximal volume (by imposing $\tilde Q \preccurlyeq \tilde P_j$ for all $j\in \mc{J}_0$, and minimizing the convex cost $-\operatorname{log} (\operatorname{det}(\tilde Q)$), where $\tilde Q \succ 0$ is a new variable). 
Similarly, $\mc{X}_0$ can  be maximized with respect to a reference shape  \cite[Eq.~20]{Thibodeau:Tong:Hu:SetInvariance:AUT:2009}. Note that in \eqref{eq:Pj} we include a positive regularization weighted by $\nu$, which allows for the inversion of $P_j$. A smaller value of $\nu$  reduces conservatism and can result in a larger guaranteed invariant set $\mc{X}_0$. \hfill $\square$ \end{remark}

\begin{remark}[Input saturation]\label{rem:inputconstraints}
Input constraints can be included in the design via additional sufficient  \glspl{LMI}. For instance, we can enforce $\|u \|_{\infty} \leq 1$ by imposing, for each $j \in \mc{J}_0$, with $L_{j,l}$ being the $l$-th row of $L_j$,
\begin{align}\label{eq:inputconstraints}
     \begin{bmatrix}
          1 & L_{j,l} \\ L_{j,l}^\top & \tilde P_j
    \end{bmatrix} \succcurlyeq 0, \quad \forall l \in \{1,2\dots,m\}.
\end{align}
 These inequalities ensure $\| K_j x \|_\infty  \leq 1$ for any $x \in \mc{E}_{P_j} \coloneqq \{x \mid x^\top P_j x \leq 1 \} \supseteq \mc{X}_0$. Interestingly, for $j\neq 0$, this condition is more conservative if $\nu$ is small. \hfill $\square$
\end{remark}

%
%
\begin{remark}[On chattering]
While the results in \cref{sec:main}-\ref{sec:Extensions:H} are based on the Lyapunov function $v_{\min}$ in \eqref{eq:vmin} (Appendix \ref{app:LM}), \cref{th:invariance} leverages $v_{\max}$ in \eqref{eq:Vmax}. An important difference, noted in \cite[p.~71]{Liberzon2003}, is that the switching rule \eqref{eq:inputswitchmax} cannot guarantee invariance or stability if sliding mode occurs --in fact, \cref{th:invariance} only considers  Carathéodory solutions. This is not a problem if $m=0$, i.e., the goal is to estimate a maximal invariant set, in the spirit of \cite{Thibodeau:Tong:Hu:SetInvariance:AUT:2009}. 
Otherwise, chattering can be avoided by using a linear controller. However, imposing $K_j = K$ for all $j$'s in \eqref{eq:maininvariance} results in nonlinear constraints (efficiently enforcing this condition is an interesting topic for future research). Furthermore, a switched controller is much more powerful than a linear gain in achieving set invariance.
Alternatively,   \cite[Th.~6]{Hu:Nonlinear:AUT:2007} constructs a  nonlinear (continuous) controller, by solving the following \glspl{BMI}: find $\{Q_j \succ 0\}_{j\in\mc{J}}$, $\beta>0$, $\Pi\in \mc{M}$ such that for all $j\in \mc{J}$,
$
    \Acl Q_j + Q_j \Acl   + \beta  Q_j - \textstyle\sum_{j\neq k}    \pi_{k,j}( Q_k - Q_j) \preccurlyeq  0.$
with $c_\ell ^\top Q_j c_\ell \leq 1$ for all $\ell \in \mc{J}$.
This inequality is already in ``dual'' form, hence a  data-driven version can be obtained via \cref{lem:S} without introducing additional nonlinearities; on the other hand, it has more variables than \eqref{eq:maininvariance}, as the matrices $Q_j$'s are not fixed. \hfill $\square$
\color{black}
\end{remark}

\section{Numerical examples}\label{sec:Numerics}
\begin{figure}[t]
		\centering
\includegraphics[width=\columnwidth]{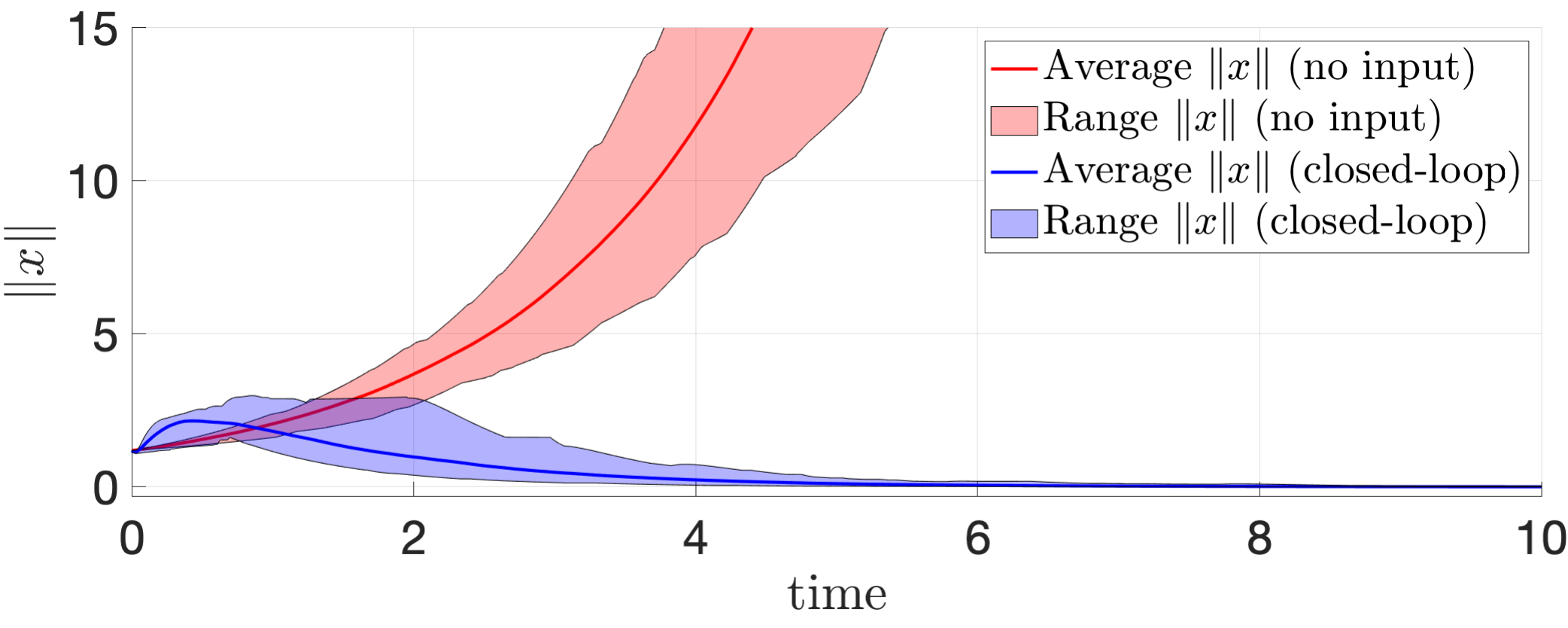}	\caption{Data-driven stabilization of a Markov jump linear system. The results are obtained by simulating 100 stochastic trajectories with the same initial condition. \label{fig:1}
}
	\end{figure}
\noindent We illustrate our results in the data-driven control of Markov jump, switched, and constrained linear systems.

\subsection{Data-driven control of Markov jump linear systems}\label{sec:numerics:MJLS}
\vspace{-0.7em}
We consider a \gls{MJLS} as in \cref{ex:Markov}, with $N=3$, $n =3$, $m = 2$,
\begin{align*}
\setlength\arraycolsep{6pt}
    \bar A_1 =\begin{smallbmatrix}
    0.5 & 0.5 & 0.3 \\ 0.1 & 0.5 & 0 \\ 0 & 0.4 & 0.3
    \end{smallbmatrix}\!,
     \bar A_2 =\begin{smallbmatrix}
    0.3 & 0.2 & 0 \\ 0 & 0 & 0 \\ 0 & 0.2 & 0.5
    \end{smallbmatrix}\!,
    \bar A_3 =\begin{smallbmatrix}
    0 & 0.1 & 0.2 \\ 0.1 &  0.5  & 0 \\ 0 & 0.1 & 0.3
    \end{smallbmatrix}\!,\\
   \bar  B_1 =\begin{smallbmatrix}
     1 & 0 \\  0 & 1\\ 0 & 0
    \end{smallbmatrix}\!,
     \bar B_2 =\begin{smallbmatrix}
     0 & 0 \\  0 & 0\\ 0 & 0
    \end{smallbmatrix}\!,
   \bar B_3 =\begin{smallbmatrix}
     1 & 0 \\  0 & 1\\ 0 & 0
    \end{smallbmatrix}\!, \Pi = \begin{smallbmatrix}
    -3 & 4 & 5 \\ 3 & -7 & 0 \\ 0 & 3 & -5
    \end{smallbmatrix}.
\end{align*}
It is easily proven that the stochastic system is open-loop unstable (cf. red plot in \cref{fig:1}). We assume that the system matrices are unknown; we simulate a trajectory with random input $u$ and measure $20$ data points for each subsystem, as in \eqref{eq:data}, with disturbance generated to satisfy \cref{asm:disturbancemodel} with $\Phi^i_{1,1} = \epsilon I$, $\Phi^i_{1,2}=0$, $\Phi^{i}_{2,2}= -I$ for all $\i$. We want to find a switching controller $ u = K_\sigma x $ that stabilizes the system by solving the \glspl{LMI} in \eqref{eq:main}; we use MATLAB equipped with YALMIP \cite{Yalmip:Lofberg:2004} (with solver Mosek). Note that the subsystem $i=2$ is not affected by the continuous input, and we assume that this information is available by first principles. To incorporate the prior knowledge in the design, we allow for different input dimensions for the modes, i.e., $m_1 = m_3 =2 $, $m_2 = 0$ (this simply correspond to changing the dimension of $L_i$ with $i$ in \eqref{eq:main}).  For $\epsilon = 10$ (corresponding to signal to noise ratio $\textnormal{SNR} = 10 \log_{10}(\|\dot X_i\|^2_{\textnormal F}/\|W_i\|^2_{\textnormal F}) \approx 25 \textnormal{dB}$), the program returns two stabilizing gains $K_1$, $K_3$; \cref{fig:1} show the resulting closed-loop behavior. For $\epsilon =20$ (SNR$\approx23$dB), the program is unfeasible: as \eqref{eq:main}
provides necessary and sufficient conditions (see \cref{ex:Markov}), this means that there exists no linear switched controller that can quadratically stabilize all the systems unfalsified by the collected data. 

\subsection{Data-driven stabilization of switched linear systems}\label{sec:numerics:switched}
\vspace{-0.5em}
\subsubsection{Stable average: solvability and computation time}\label{sec:numerics:switched:A}
\vspace{-0.5em}
We consider a switched linear system as \eqref{eq:system}, with $N=3$, $n=2$, $m=0$, 
\begin{align*}
       \bar A_1 =\begin{smallbmatrix}
         2 & 0.1 \\  0.1 & -0.2
    \end{smallbmatrix}, 
    \bar A_2 =\begin{smallbmatrix}
         -10 & 0.1 \\  0.1 & 0
    \end{smallbmatrix}, 
        \bar A_3 =\begin{smallbmatrix}
         0.1 & 0\\  0 &  0.1
         \end{smallbmatrix}.
         \end{align*}
Each mode is unstable. 
We note that, even in the case of perfectly known model, the  inequality in \eqref{eq:LM_modified} (i.e., the relaxation proposed in \cite{GeromelColaneri_SIAM2006}) does not admit a solution\footnote{In fact, due to the particular form of  $\bar A_3$, for $i=3$ and $j =1$, \eqref{eq:LM_modified} would imply that $(P_j-P_i)$ is negative definite; in turn, \eqref{eq:LM_modified} for $i=1$ and  $j = 3$ reduces to the Lyapunov inequality and would imply that $A_1$ is Hurwitz, which is false.}. Yet, $\frac{1}{3} (\bar A_1+ \bar A_2+ \bar A_3)$ is Hurwitz; thus, the \gls{LM} inequalities  \eqref{eq:LM1} are solvable with $\Pi$ as in \eqref{eq:specialPi_average}, and so must be \eqref{eq:main} for sufficiently informative data/small disturbances, by \cref{lem:existenceofspecialsol}.

We perform simulations for several noise bounds; let
\begin{align}\label{eq:Phiepsilon}
\Phi^i_{1,1} =  \epsilon I, \ \Phi^i_{1,2}=0, \  \Phi^{i}_{2,2}= -I, \quad (\forall \i),
\end{align}
 for different values of $\epsilon \in \{0.1, 1, 10, 20, 40, 80 \} $. For each disturbance level, we generate $100$ datasets satisfying \eqref{eq:data} (where disturbances are randomly generated and  normalized such that $ \| \bar W_i \bar W_i^\top \| = \upsilon \epsilon $, $\upsilon$ uniformly drawn in $[0,1]$, to satisfy \cref{asm:disturbancemodel}), with $T_i = 20$, for all $i$,; the Slater's condition \eqref{eq:Slaters} is verified for all datasets. Our goal is to find suitable matrices $P_i$'s to implement the stabilizing controller \eqref{eq:control}; in particular, we investigate feasibility of \eqref{eq:main2} for $\Pi$ restricted to be  as in \eqref{eq:specialPi_average} and as in \eqref{eq:specialPi_convex}, respectively. In both cases, we solve \eqref{eq:main2} on Matlab, using Yalmip (solver Mosek) and line-search (over one scalar $\gamma$ in the first case, over three scalars $\{\gamma_i\}_{\i}$ in the second; in both cases, between $1$ and $100$). We record the percentage of  experiments for which we could find a solution, and the corresponding solver time on a commercial laptop. The results are shown in \cref{tab:1} (excluding the fourth and seventh columns).

\begin{table}[t] 
{\small
	\setlength{\tabcolsep}{4pt}
	\centering
	\newlength{\mywidth}
	\setlength{\mywidth}{0.8em}
	\begin{tabular}{ c@{\hskip 2em} c c c@{\hskip 2em} c c c }
		\toprule
		\multirow{2}{*}{$\epsilon$}
		&
		\multicolumn{3}{l}{Solvability of \eqref{eq:main2} (\%)}
		&
		\multicolumn{3}{l}{Average solver time (s) \hspace{-1em}}
		\\  
		& \begin{tabular}{c}
		     \eqref{eq:specialPi_average}
       \\[-0.5em] LS
		\end{tabular}
	    &	
     \begin{tabular}{c}
     \eqref{eq:specialPi_convex}
       \\[-0.5em] LS
		\end{tabular}
	    &  	
     \glsunset{RM} 
       \begin{tabular}{c}
     \eqref{eq:specialPi_convex}
       \\[-0.5em] \gls{RM}
		\end{tabular}
  & \begin{tabular}{c}
		     \eqref{eq:specialPi_average}
       \\[-0.5em] LS
		\end{tabular}
	    &	
     \begin{tabular}{c}
     \eqref{eq:specialPi_convex}
       \\[-0.5em] LS
		\end{tabular}
	    &  	
     \glsunset{RM} 
       \begin{tabular}{c}
     \eqref{eq:specialPi_convex}
       \\[-0.5em] \gls{RM}
		\end{tabular}
  \glsreset{RM}%
		\\
		\midrule 
		0.1 & 100 & 100 & 100  &  0.67 & 3.64 & 0.25 \\ 
		1 & 100 & 100 & 100  & 0.60 &  3.45 & 0.25 \\ 
		10 & 100 & 100 & 100 & 0.68 &  12.31 & 0.44 \\ 
	    20 &  71 & 100 & 100 & 0.92 & 21.81 & 0.30 \\ 
		40 & 0 & 100 & 100 &  -- & 47.57 & 0.53 \\ 
		80 & 0 & 69 & 57 & -- & 102.14 &  3.45 \\ 
		 \bottomrule
	\end{tabular}
	\vspace{0.3em}	\caption{\label{tab:1} Percentage of solutions found and total computation time of \eqref{eq:main2}, for the two relaxations \eqref{eq:specialPi_convex} and \eqref{eq:specialPi_average}, for  several noise levels $\epsilon$, and employing either line-search (LS) or \glsfirst{RM} in the solver.}
	}
\end{table}

As expected, both conditions admit a solution in all the experiments, for small noise levels. The second condition is by far computationally more expensive, but also less restrictive. In fact, for larger noise bounds (i.e.,  larger compatibility set) imposing \eqref{eq:specialPi_average} becomes too conservative, resulting in unfeasibility --even though the average of the \emph{true} system matrices is Hurwitz.
Intuitively, the gap between the two conditions, in terms of both complexity and conservatism, is bound to further grow if the number of modes $N$  (equivalently, the number of free parameters in \eqref{eq:specialPi_convex}) increases. This is a fundamental trade-off, which also arises in the model-based case. 

Alternatively,  we  also seek a solution to \eqref{eq:main2} under \eqref{eq:specialPi_convex} via a \gls{RM} heuristic that, while giving up on theoretical guarantees, exhibits good empirical  performance.  In particular, we exploit the reformulation \eqref{eq:rankcondition}, and we attempt to solve it via the reweighted nuclear norm iteration in \cite{Boyd:RankMinimization:ACC:2003} --recently successfully employed in the context of data-driven control in \cite{Dai_Semialgebraic_LCSS2020,Dai_Switched_CDC2018}; the algorithm consists of a sequence of \glspl{SDP}, which we solve with YALMIP \cite{Yalmip:Lofberg:2004} and Mosek. Note that, although computing a solution to \eqref{eq:main2} is hard in general, verifying a solution (from data only) is straightforward. The results are shown in the remaining columns of \cref{tab:1}. The average solver time improves remarkably with respect to the brute-force line-search approach (even if we compare it with the more conservative case \eqref{eq:specialPi_average}). However, for the largest noise level, the \gls{RM} method fails to return a solution for some experiments  where line-search succeeds.
%
%

\subsubsection{Stable convex combination and $\mc{H}_\infty$ performance}\label{sec:numerics:switched:B}
\vspace{-0.5em}

We consider a switched system with disturbances in the form \eqref{eq:system2}, with $N=3$, $n=3$, $m=0$, 
\begin{align*}
    \bar A_1 &= \begin{smallbmatrix}
          -1 & -0.1 & 0.1 \\
          0.1 & 0.1 & 0.1 \\
          -0.1 & -0.1 & 0.1
    \end{smallbmatrix}, \quad 
    E_1 = \begin{smallbmatrix}
          0 & 0 \\ 1 & 0 \\ 0 & 1
    \end{smallbmatrix}
    \\
    \bar A_2 & =
    \begin{smallbmatrix}
          0.1 & -0.1 & 0.1 \\
          0.1 & -0.1 & 0 \\ 
          -0.1 &  0 & 0.1
    \end{smallbmatrix}, \quad E_2 = \begin{smallbmatrix} 
          1 &  0 \\ 0 &  0 \\ 0 & 1
    \end{smallbmatrix}
    \\ 
        \bar A_3 & = \begin{smallbmatrix}
          0.1 &  0.1 & 0.1 \\
          -0.1 & 0.1 & -0.1 \\
          -0.1 & 0.1 &  -1
    \end{smallbmatrix}, \quad E_3 = \begin{smallbmatrix}
          1 & 0 \\
          0 & 1 \\
          0 & 0
    \end{smallbmatrix},
\end{align*}
and, for all $\i$, $F_i = 0$ and $C_i = \diag(1,3,1)$, corresponding to a larger penalization for the second state.  We collect $100$ datasets from open-loop experiments, with $T_i=20$ for all $\i$, and the samples of the disturbance $\dist$  satisfying an energy bound as in  \eqref{eq:Phiepsilon}; the corresponding compatibility sets and matrices $\Phi_i$'s are computed as in \cref{rem:refineddisturbance}. For each dataset,  we aim at designing a controller to optimize the $\mc{H}_\infty$ performance of the system, by solving the data-driven program \eqref{eq:mainHinf}.
Note that each $\bar A_i$ is unstable and so is $\frac{1}{3}(\bar A_1+\bar A_2+\bar A_3)$; moreover, even in the case of known model, the condition \eqref{eq:specialPi_average} results in unfeasibility of the \gls{LM} inequalities \eqref{eq:LM1} --and similarly of \eqref{eq:mainHinf}. However, with an oracle of the system matrices, feasibility of  \eqref{eq:mainHinf} under \eqref{eq:specialPi_convex} is expected when the compatibility sets are  small enough, based on \cref{lem:existenceofspecialsol} and the fact that the matrix $0.2 \bar A_1+0.6\bar A_2+0.2\bar A_3$ is Hurwitz.

In fact, for small noise bounds, we are able to find a solution to \eqref{eq:mainHinf} for all the datasets, via line-search over $\{\gamma_i\}_{\i}$. For any fixed value of these parameters, minimizing $\rho$ subject to \eqref{eq:specialPi_convex} is an \gls{SDP}. The average optimal value obtained for disturbance bound $\epsilon = 0.001$, $0.01$, $0.1$ (i.e., $\textnormal{SNR}\approx
 63$dB, $54$dB, $38$dB, respectively) is $\rho = 45.2$, $55.4$, $150.4$, respectively. For $\epsilon =1$ ($\textnormal{SNR}\approx 33$dB), the problem is unfeasible; the optimal value assuming perfect knowledge of the model is $41.6$. 
 Finally, to evaluate the impact on performance when the noise bound is overapproximated, we generate datasets with disturbance model in \eqref{eq:Phiepsilon} and $\epsilon = 0.001$, but solve \eqref{eq:mainHinf} by using the conservative bound  $\Phi^i_{1,1} =  \hat \epsilon I$, for some $\hat \epsilon > \epsilon$. The average guaranteed $\mc{H}_\infty$ performance for $\hat \epsilon = 0.01$, $0.1$, $0.2$, $0.3$  is $\rho = 54.3$, $131.1$, $612.6$, $1462.3$, respectively;  for $\hat \epsilon = 1$ the problem is always infeasible. As expected, the performance deteriorates, both with larger disturbances and coarser bounds. \color{black}

\cref{fig:2} shows one closed-loop trajectory under the controller \eqref{eq:control}, with disturbance $\dist(t) = \frac{1}{t} \operatorname{col}(\sin(t),\sin(t-\frac{2}{3}\pi),\sin(t-\frac{4}{3}\pi))$. The bottom-right plot shows part of the switching signal: after a transitory where  mode $i=2$ is active, the trajectory hits a sliding surface and starts chattering between modes $2$ and $3$ (a case we accounted for in our analysis,  see \cref{prop:Hinf} in \cref{app:LM}).



	\begin{figure}[t]
	        \vspace{-1.1em}
			\hspace{-1em}
\includegraphics[width=1\columnwidth]{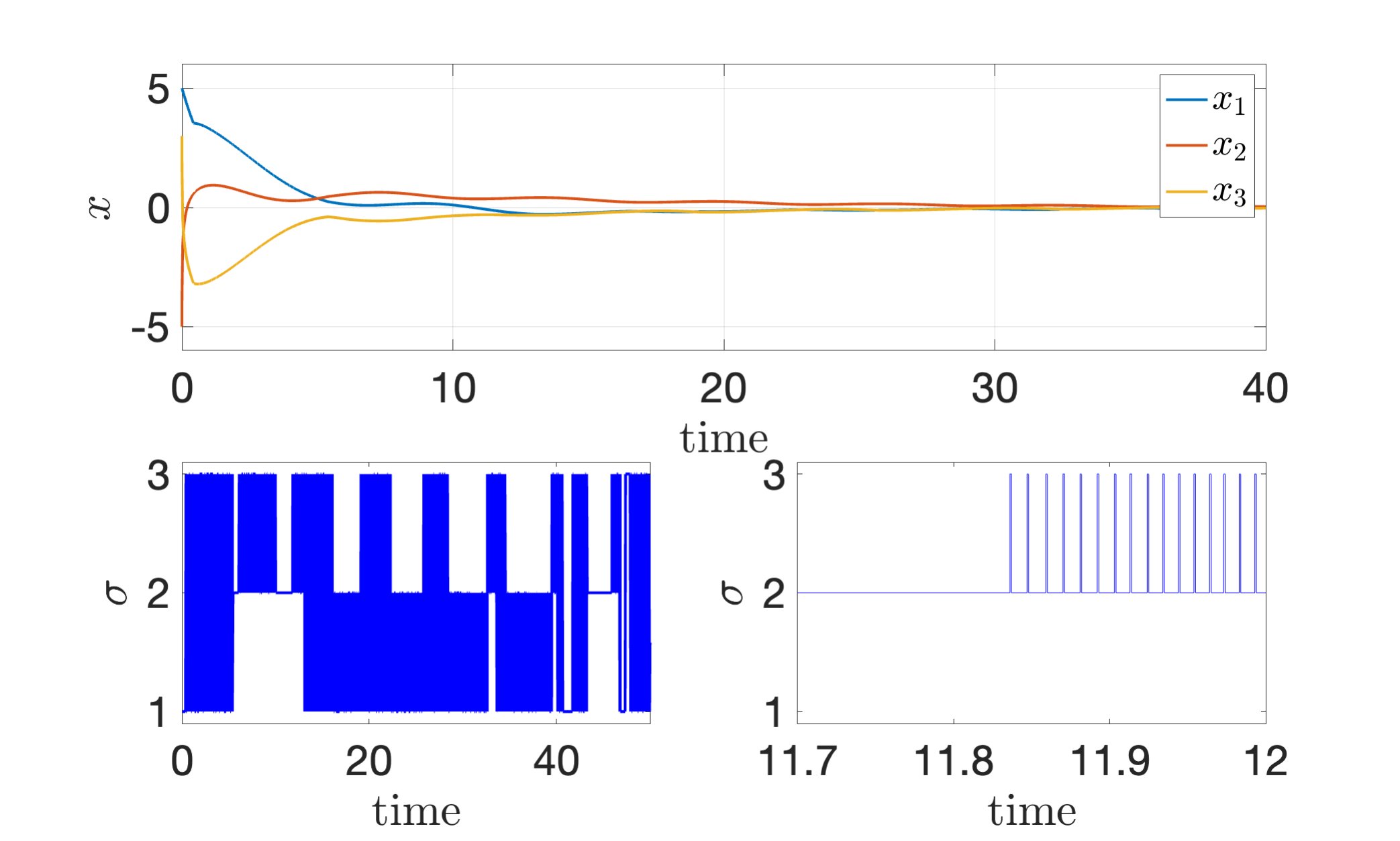}
\caption{$\mc{H}_\infty$ data-driven stabilization of a switched linear system with disturbances: state evolution (top), and switching signal (bottom-left, and a detail, bottom-right). Importantly, \cref{th:main:Hinf}  guarantees the  $\mc{H}_\infty$ performance of the real unknown system even when chattering occurs, as  in this example.\label{fig:2}
}
\end{figure}



\subsection{Data-driven robust constrained stabilization}
\vspace{-0.7em}

Consider an \gls{LTI} system as in \eqref{eq:disturbedLTI}, with $n= 3$, $m = 1$,
\begin{align*}
    \bar A = \begin{smallbmatrix}
          -1 & 0 & 0 \\ 1 & -2 &-1\\ 0 & 1 & 0
    \end{smallbmatrix}, \quad \bar B = \begin{smallbmatrix}
          0 \\ 0 \\1
    \end{smallbmatrix}, \quad E =  I_3,
\end{align*}
where the disturbance $\dist$ satisfies the instantaneous bound $\dist(t)^\top \dist(t) \leq \eta $,
subject to the state constraints $x \in \mc{X} \coloneqq \{ x\mid \|x \| _\infty \leq 1\}$ (i.e., \cite[Ex.~4]{Thibodeau:Tong:Hu:SetInvariance:AUT:2009}). 

The system matrices $\bar A$, $\bar B$ are unknown, but an open-loop experiment is recorded, of length $T =20$; the data collected satisfy \cref{asm:disturbancemodel} with $\Phi =\Phi_1$ as in \eqref{eq:Phiepsilon}, $\epsilon = T \eta  $  (see  \cref{rem:alternative_disturbance}). 
We leverage \cref{th:invariance} to solve the robust constrained stabilization problem, 
by additionally imposing the input constraints
$\|u \|_\infty \leq 1$, via the conditions in  \eqref{eq:inputconstraints}.

	\begin{figure}[t]
		\centering
\includegraphics[width=0.95\columnwidth]{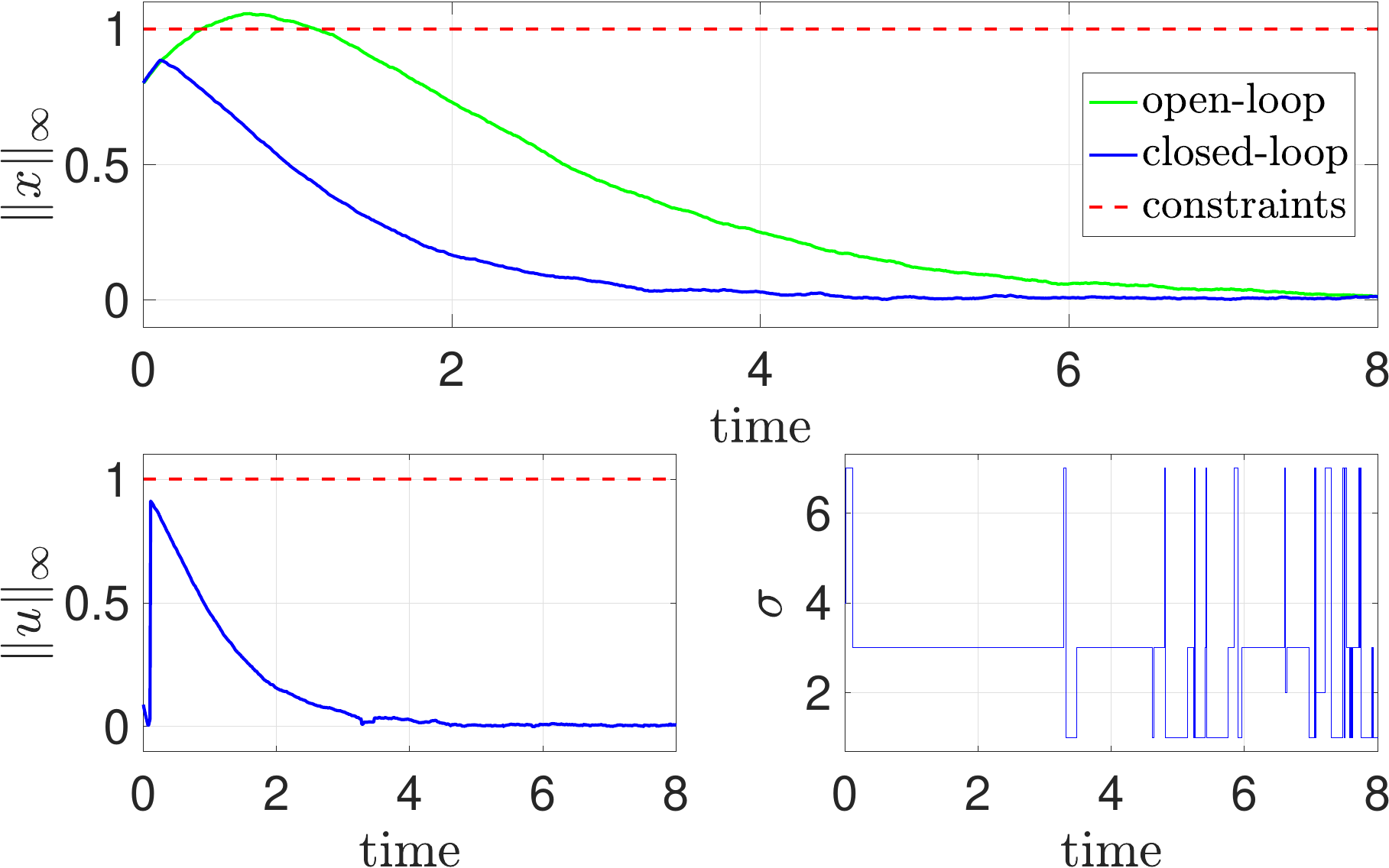}	\caption{Data-driven constrained robust stabilization of a disturbed \gls{LTI} system, via a switched compensator (with active mode $\sigma$). \label{fig:invariance}
}
\end{figure}

We fix $\nu = 0.1$ in \eqref{eq:Pj}. 
To maximize the invariant set $\mc{X}_0$, we look for the smallest $\theta>0 $ such that $\theta^{-\frac{1}{2}}\mc{X}  \subseteq \mc{X}_0$; the latter condition is enforced via the  \glspl{LMI} $  v_k ^\top P_j v_k \leq \theta $, for each $j\in \mc{J}_0$ and each vertex $v_k$ of $\mc{X}$ \cite{Thibodeau:Tong:Hu:SetInvariance:AUT:2009}. 
We also note that the
data-driven inequalities \eqref{eq:maininvariance2} can be recast as a rank-constrained \gls{LMI} (similarly to \eqref{eq:rankconstrained}), by imposing $\operatorname{rank}(M) = 1$, where $M = h h^\top$ and $h = \begin{bmatrix}
      \tilde \gamma_0 &   \beta_0 &  \operatorname{vec}(\tilde P_0)^\top  & \operatorname{vec}(P_0)^\top
\end{bmatrix}$. 
We then solve the resulting data-driven problem via the algorithm in \cite{Boyd:RankMinimization:ACC:2003} and bisection over $\theta$, to design the switched controller \eqref{eq:inputswitchmax} and the invariant set $\mc{X}_0$. 

 \cref{fig:invariance} show a closed-loop simulation for the case $\eta = 0.1$ ($\textnormal{SNR} \approx 32$dB), with randomly generated disturbances.  The top plot shows that the closed-loop system achieves safety, while the uncontrolled system violates the constraints (here, $x(0)= \col(0.8, 0.8,0.8) \in \mc{X}_0$). The volume of the obtained set $\mc{X}_0$ is $75 \% $  of the volume of  $\mc{X}$. For comparison, the percentage obtained with $\nu =  1$ and $\nu=0.01$ are $24 \%$ and $77 \%$, respectively; for $\nu=0.001$, the problem is unfeasible (see \cref{rem:inputconstraints}).

Next, in \cref{fig:volumes}, for different values of the disturbance bound $\eta$, we compare the volume of the guaranteed invariant set $\mc{X}_0$ obtained with our switched controller, in \eqref{eq:inputswitchmax}, with that obtained via:
\begin{enumerate}[leftmargin=*,topsep=-1em]
\item  a data-driven linear controller $u = Kx$, designed as in \cref{sec:Constrained}, by replacing the definition of the Lyapunov function $v_{\max}$ in \eqref{eq:Vmax} with the quadratic $v_{\max} = x^\top P_0 x$, where $\mc{X}_0 \subseteq \mc{X}$ as in \eqref{eq:X0} is imposed via the additional constraints $\tilde P_0 \preccurlyeq \tilde P_j$, for all $j \in \mc{J}$, and the volume of $\mc{X}_0$ is maximized as in \cref{rem:maxX0};
\item the data-driven polynomial controller in \cite{Luppi:Bisoffi:DePersis:Tesi:SafePolynomial:arXiv2021}.

\end{enumerate}
 Note that a larger disturbance bound affects not only the quality of data, but also the problem itself (i.e., ensuring robust invariance is more challenging; mathematically, the matrix $E$ in \eqref{eq:maininvariance2} is scaled). 
 
 We first compare the results obtained by neglecting the input constraints (top axes). For the switched and the linear controller, we choose the regularization parameter $\nu = 10^{-3}$. For the small disturbance bound $\eta = 0.01$, the switched controller achieves a guaranteed invariant set  $\mc{X}_0$  with volume $97 \%$ of the volume of $\mc{X}$, namely almost all constraint set is guaranteed to be invariant under the switched controller in \eqref{eq:inputswitchmax}. We also remark that the 
  polynomial controller only ensures nominal invariance, i.e., invariance in the  absence of  closed-loop disturbance, hence the size of the obtained $\mc{X}_0$ decreases less for increasing $\eta$ (the condition in \cite[Rem.1]{Luppi:Bisoffi:DePersis:Tesi:SafePolynomial:arXiv2021}, which could be used to enforce robust invariance, resulted in infeasibility in our simulations).

In the bottom plot of \cref{fig:invariance} we instead also consider input constraints. For the largest noise bound $\eta = 2$, the program \eqref{eq:maininvariance2}-\eqref{eq:inputconstraints} fails to return a solution. 
 For the polynomial controller, the  input constraints  imposed as in \cite[Rem.~4]{Luppi:Bisoffi:DePersis:Tesi:SafePolynomial:arXiv2021} always result in infeasibility in our experiments.

%

%
%


\begin{figure}[t]
	\centering
\includegraphics[width=\columnwidth]{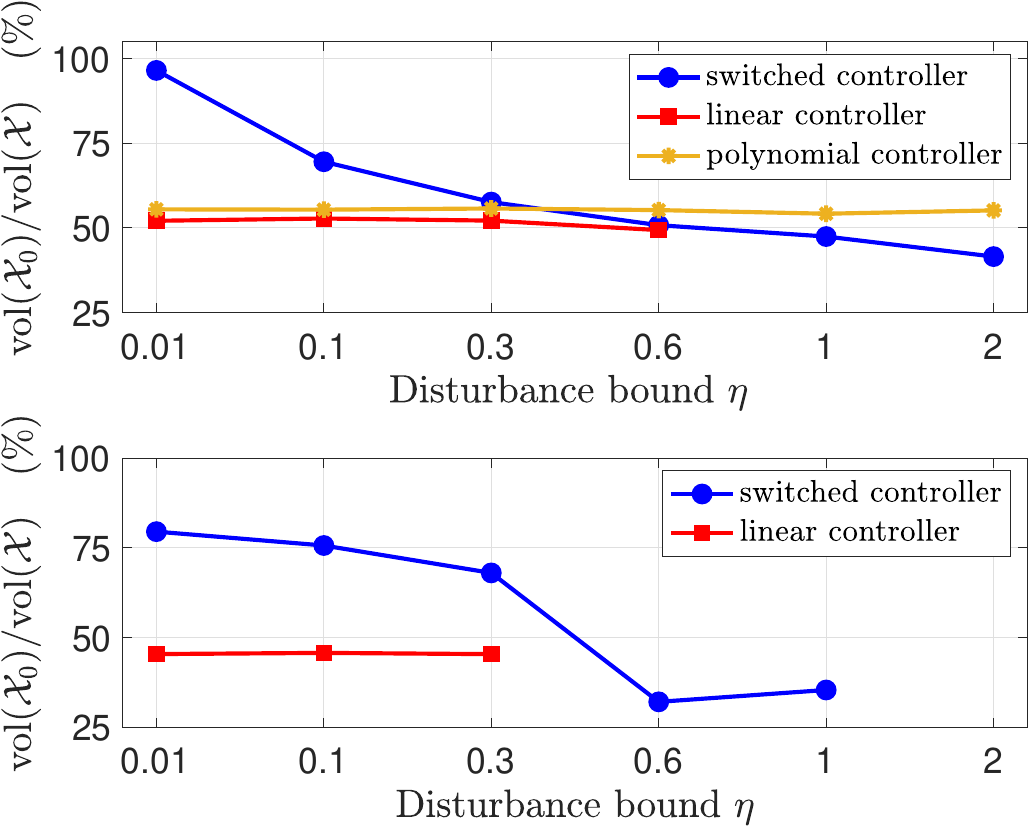}
\caption{Volume of the safe invariant set $\mc{X}_0$ obtained with data-driven compensators: we compare the proposed switched controller versus a linear controller and the polynomial controller in \cite{Luppi:Bisoffi:DePersis:Tesi:SafePolynomial:arXiv2021}. The input is unconstrained in the top plot and constrained as  $\|u\|_{\infty}\leq 1$ in the bottom plot. Missing data-point represent infeasibility of the corresponiding data-based programs. The linear and switched controllers ensure stability and  robust invariance; the polynomial controller  guarantees nominal invariance.  \label{fig:volumes}} 
\end{figure}

\subsection{Data-driven viral escape mitigation}
\vspace{-0.7em}

Here, we focus on treatment scheduling for viral escape mitigation. We consider the switched linear virus mutation dynamics of \cite{Blanchini_CTViralEscape_2010}. The virus has $n$ different genotypes; $x_\ell \coloneqq [x]_\ell$ is the viral population of the $\ell$-th genotype. At time $t$, the drug $\sigma(t) \in \{1\,2,\dots,N\}$ is administered, chosen among $N$ different therapies, more or less effective against each genotype: drug $i$ results in the proliferation rate $[R]_{\ell,i}$ for genotype $\ell$. The resulting dynamics are
\begin{equation}\label{eq:virussystem}
    \dot x_{\ell} = \left( [R]_{\ell,\sigma}  - \delta \right) x_\ell + \mu \textstyle \sum_{j\in\mc{I}, j \neq \ell} [M]_{\ell,j} x_j, 
\end{equation}
where $\delta>0$ is the clearing rate, $\mu>0$ the mutation rate, and $[M]_{\ell,j} =1$ if genotype $j$ can mutate into genotype $\ell$, $0$ otherwise. Here, $\mu = 10^{-4}$, $\delta =0.24 $, $n=5$, $N=4$,
\begin{align}
    R = \begin{smallbmatrix}
        0.05 & 0.05 & 0.37 & 0.25 
        \\
        0.39 & 0.05 & 0.21 & 0.14 
        \\
        0.05 & 0.39 & 0.05 & 0.23
        \\
        0.30 & 0.14 & 0.21 &0.22
        \\
        0.27 & 0.09 & 0.21 & 0.04
    \end{smallbmatrix}, \quad  
    M = \begin{smallbmatrix}
        0 & 1 & 1 & 1 & 0
        \\
        1 &0 &0 &0 &1
        \\
        1 &0 &0 &0 &1
        \\
        1 &0 &0 &0 &1
        \\ 
        0 &1 &1 &1 &0
    \end{smallbmatrix}.
\end{align}
\indent The system in \eqref{eq:virussystem} is an instance of \eqref{eq:system} with $m=0$. We assume that the system is unknown, but data on the effectiveness of each therapy have been recorded
in the form of \eqref{eq:data}, with $T_i=T$ for all $i\in\mc{I}$. The data are affected by a disturbance with bounded sample covariance, namely, for all $i$, $\frac{1}{T} \sum_{\tau= 1}^{T}  (\noise(t_\tau^ i)- \noise_{\text{avg}}^i)(\noise(t_\tau^i)-\noise_{\text{avg}}^i)^\top \leq 0.02 I$, with $\noise_{\text{avg}}^i = \frac{1}{T} \sum_{\tau= 1}^{T} \noise(t_\tau^ i) $. Equivalently, \cref{asm:disturbancemodel} is satisfied with $\Phi^i_{1,2} = \0$, $\Phi_{1,1}^i= 0.02 I$, $\Phi_{2,2} = I - \frac{1}{T} \1 \1^\top$. 
The objective is to choose the discrete input $\sigma$ to overcome the infection as quickly as possible. Thus, we look for a switched controller as in \eqref{eq:control}, by imposing the data-driven condition in \eqref{eq:main2}, and by maximizing the worst-case closed-loop convergence rate as per \cref{rem:rate}.\footnote{
 A drug-resistant genotype could also be studied \cite{Blanchini_CTViralEscape_2010}. In this case, the system would be nonstabilizable, meaning that one can only hope to slow down the infection as much as possible. The latter problem can be cast in terms of data by modifying \eqref{eq:main2} via an exponential discounted change of coordinates \cite{Blanchini_CTViralEscape_2010}.}
\newline\indent
\cref{fig:rates} shows the rates obtained for different values of $T$ 
($\textnormal{SNR} \approx 32$dB in all cases; for $T$ smaller than $50$, no solution is found), and highlights the benefit of having larger datasets, see \cref{rem:batchsizes}. We remark that the implementation of the controller \eqref{eq:control} could result in chattering, which might be undesirable in some cases. One way to avoid this issue is to choose $\sigma$  according to a Markov chain with transition matrix $\Pi^\top$ ($\Pi$ as in \eqref{eq:main2}). This open-loop controller still-ensures (stochastic) stability (see \cref{ex:Markov}) and further does not require state measurement, possibly at the cost of some performance, as illustrated in \cref{fig:virus_trajectory}. 

 \begin{figure}[t]
			\hspace{-1em}
\includegraphics[width=1\columnwidth]{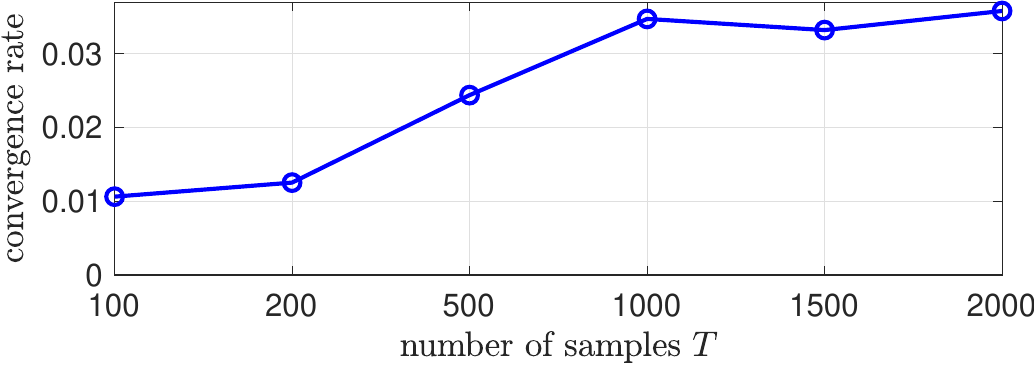}
\caption{ Guaranteed convergence rate for different dataset sizes. \label{fig:rates}
}
\end{figure}
 \begin{figure}[t]
			\hspace{-1em}
\includegraphics[width=1\columnwidth]{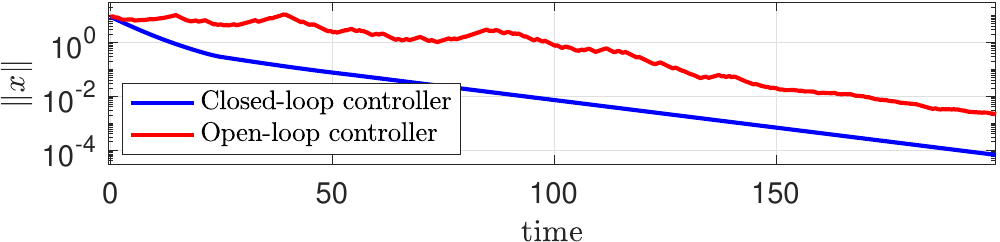}
\caption{ Viral load, with drug therapy scheduling chosen based on the current state or in open-loop according to a Markov chain.  \label{fig:virus_trajectory}
}
\end{figure}

\color{black}
\vspace{-0.5em}
\section{Conclusion}
\vspace{-0.5em}

We have considered the stabilization of an unknown switched linear systems by solving data-driven Lyapunov-Metzler inequalities, parameterized by the set of available noisy experimental data.
Since the problem is computationally expensive even for  a known system, it is crucial to massage the matrix conditions to allow for efficient solution. We have presented various relaxations that significantly reduce the computational cost.
Our techniques find application beyond control of switched systems, for instance in data-driven constrained stabilization. More generally, they can be used to recast in a data-driven fashion (without introducing conservatism nor additional computational complexity)
a large class of coupled-Lyapunov inequalities,  pervasive in problems related to switching systems and  stabilization of differential inclusions   \cite{Hu:Ma:Lin:Composite:TAC:2008}, \cite{Hu:Blanchini:Nonconservative:AUT:2010}.

As future work, it would be valuable to address the data-based design of \emph{mode-independent} continuous compensators, for switched or Markov jump systems. \emph{Nonconservative} design of stabilizing controllers  based on input-\emph{output} data is also a prominent open challenge. The extension of our results to \emph{discrete-time} switched linear systems is left for future research.


\appendix

\section{\gls{LM} inequalities and sliding motion}\label{app:LM}

We consider a continuous-time switched linear system 
\begin{align}\label{eq:system_nou}
    \dot{x}(t) = A_{\sigma(t)} x(t),
\end{align}
where $x\in \R^{n}$, $\sigma(t)\in \mc{I}=\{1,2,\dots,N\}$ is a controlled switching signal, $\{A_i\}_\i$ are the system matrices. We are interested in the \emph{min-switching} feedback law
\begin{align}\label{eq:switching_control}
    \sigma(x) =  \min \left\{ \textstyle \argmin_{\i} \ x^\top P_i x \right\},
\end{align}
(the $\min$ selects the minimum index when the $\argmin$ is set-valued; any other selection rule can also be chosen), where the matrices $\{P_i \in \R^{n\times n} \}_{\i
}$ solve the following \gls{LM} inequalities problem \cite{GeromelColaneri_SIAM2006}:
find $\{P_i \succ 0 \}_{\i}$, $\{Q_i \succcurlyeq 0\}_{\i}$, $\Pi =[\pi_{i,j}]_{i,j} \in \M$ such that 
\begin{align}\label{eq:LM_original}
 \hspace{-1em} \forall  \i,  \quad  A_i^\top P_i + P_i A_i + \textstyle \sum_{j\in \mc{I}} \pi_{j,i} P_j +Q_i \prec 0.
\end{align}

\subsection{Solution concept and stability}
It was shown in \cite{GeromelColaneri_SIAM2006} that 
matrices $\{P_i\}_{\i}$ satisfying \eqref{eq:LM_original} ensure asymptotic stability for all  Carathéodory solutions of the closed loop system \eqref{eq:system_nou}-\eqref{eq:switching_control}. However,  the switching rule \eqref{eq:switching_control} can results in \emph{chattering}. For this reason, we instead consider the \emph{Filippov}\footnote{\eqref{eq:sliding} actually defines a superset of the Filippov solutions, as we do not exclude sets of measure zero (which would require  $\mathrm{I}(x)  \coloneqq \{\i \mid \forall  V \in \mc{N}(x), \exists U \subseteq V \cap \mc{X}_i \text{ s.t. } \mu(U)>0\}$); in this way we also include all Carathéodory solutions.} solutions
of \eqref{eq:system_nou}, namely  absolutely continuous trajectories $x:\R_{\geq 0} \rightarrow \R^n$ such that, for almost all $t$, 
\begin{align} \label{eq:sliding}
    \dot x(t) = \textstyle \sum_{i\in \textrm{I(x)}} \alpha_i(x) A_i x,
\end{align}
for some $\{\alpha_i(x)\} _{i \in I(x)} \in \Delta$, where 
\begin{align}
    \mathrm{I}(x) & \coloneqq \{\i \mid \forall  V \in \mc{N}(x), \exists y \in V \text{ s.t. } y\in \mc{X}_i \}  
    \\
    \mc{X}_i &\coloneqq \{ x \in \R^{n} \mid \sigma(x) = i \},
    \end{align}
    where $\mc{N}(x)$ is the set of neighborhoods of $x$ (i.e., the set of all open subsets of $\R^n$ containing $x$). When $x \in \operatorname{int}(\mc{X}_i)$,  $\mathrm{I}(x) = \{i\}$ is a singleton and  $\dot x  = A_i x_i$. A solution can also cross the boundary between two regions $\mc{X}_i$ and $\mc{X}_j$. Finally, $x(t)$ can evolve along the boundaries between two or more regions, in a direction specified not by one of the modes $\i$, but by a convex combination of the matrices $A_i$ with $i \in \mathrm{I(x)}$: in this case we talk about ``sliding mode''. Although ideal sliding mode would not happen in practice (due to discretized controllers, hysteresis, time-delay), it provides a close approximation of the behavior of the real system under fast switching. 
    
It is known that the \gls{LM} inequalities do \emph{not} ensure stability for all closed-loop Filippov solutions: in fact, \emph{repulsive} sliding motion \cite{Filippov:1988} can cause instability \cite[Ex.~1.1]{Heemels:Weiland:ISSdiscontinuous:AUT:2008}, \cite[Rem.~6]{Heemels:Kundu:Daafouz:LMinequalities:TAC2017}. On the other hand, repulsive sliding mode would not appear in practice, e.g., if the controller is discretized, (and further implies the existence of alternative solutions to \eqref{eq:system_nou}). For the case of $N=2$, asymptotic stability of all solutions with \emph{attractive} sliding mode  was shown in \cite[p.~70]{Liberzon2003}. The result was generalized in \cite[Rem.~2]{GeromelColaneri_SIAM2006} and \cite{Geromel:Colaneri:Hsu:Solucoes:Brasileiro:2018} to all Filippov solutions  such that\footnote{ \cite{GeromelColaneri_SIAM2006,Geromel:Colaneri:Hsu:Solucoes:Brasileiro:2018} actually assume the stronger condition $\dot   v_{\sigma(x)} ( x, \alpha_i(x) A_i x ) \leq \dot v_{i} (x, \alpha_i(x) A_ix)$ $\forall i \in \textrm{I}(x)$.},
for almost all $t$
\begin{align}\label{eq:asmsliding}
    \dot   v_{\sigma(x)} \left( x,\textstyle \sum_{i\in \textrm{I}(x)} \alpha_i(x) A_i x \right) \leq \textstyle \sum_{i\in \textrm{I}(x)} \dot v_{i} (x, \alpha_i(x) A_ix)
\end{align}
with (possibly zero) $\alpha_i$'s as in \eqref{eq:sliding}, and \begin{align}
    \dot v_i (x, \xi) \coloneqq x^\top P_i \xi +\xi^\top P_i x \end{align} 
is the directional derivative of $v_i(x) \coloneqq x^\top P_i x$ at $x$ along $\xi$; nevertheless, these results do not take into account performance specifications. We remedy this in the following, by  restricting our attention to the same class of trajectories: throughout the paper, by solution of \eqref{eq:system_nou} we mean a trajectory satisfying \eqref{eq:sliding}--\eqref{eq:asmsliding}\footnote{\label{foot:slidingcondition}With analogous definition when considering systems with input or disturbances: if \eqref{eq:system_nou} is replaced by $\dot x(t) = \xi_{\sigma(x)} (x,t)$ for some (time dependent) mappings $\{\xi_i\}_{\i}$, then the term $A_i x$ shall be replaced by $\xi_i (x,t)$ in \eqref{eq:sliding} and \eqref{eq:asmsliding}. The considerations in \cref{rem:sol_concept} are still valid.}.

\begin{remark}[On the solution concept] \label{rem:sol_concept} Condition \eqref{eq:asmsliding}  virtually always holds in practice: for example, it is verified for any attractive sliding motion involving only two modes \cite[Eq.~3.22]{Liberzon2003} -- the most relevant case and often the only considered \cite{Liberzon2003}, \cite{Filippov:1988}. Indeed, the proof of \cite[Prop.~1]{Hu:Ma:Lin:Composite:TAC:2008} argues that 
\eqref{eq:asmsliding} is \emph{necessary} for the occurrence of chattering if \eqref{eq:switching_control} is discretized with arbitrarily small sampling time. Nonetheless, examples can be constructed where the continuous-time system \eqref{eq:system_nou} does not admit any solution satisfying \eqref{eq:asmsliding} (at the boundary between three or more regions $\mc{X}_i$). These pathological cases are excluded from our analysis: like the related literature \cite{GeromelColaneri_SIAM2006,Hu:Ma:Lin:Composite:TAC:2008}, we assume throughout existence of a solution satisfying \eqref{eq:sliding}-\eqref{eq:asmsliding}, wherever needed.  \hfill $\square$
\end{remark}
Following \cite{GeromelColaneri_SIAM2006}, we study the stability of \eqref{eq:system_nou} via the non-convex, non-differentiable Lyapunov function
\begin{align}\label{eq:vmin}
    v_{\min}(x) \coloneqq \min_{\i} \  x^\top P_i x = \min_{\i} v_i(x),
\end{align}
with matrices $\{P_i\}_{\i}$ solving \eqref{eq:LM_original}. 

\setpropositiontag{A1}
\begin{proposition}[Switched stabilization] \label{prop:LMstability}
    If there exist  $\{P_i\succ 0, Q_i\succcurlyeq 0\}_{\i}$ and  $\Pi=[\pi_{i,j}]_{i,j \in \mc{I}}  \in \M$ satisfying the Lyapunov-Metzler inequalities \eqref{eq:LM_original}, 
    then the switched feedback control law \eqref{eq:switching_control} makes $x^*=0$ globally asymptotically stable for the system \eqref{eq:system_nou}. Moreover, it holds that
    \begin{align}\label{eq:performance}
       \textstyle   \int_{0}^{\infty} \ x^\top Q_{\alpha(x)} x \ \text{d}t< \min_{\i} \ x(0)^\top P_i x(0), 
    \end{align}
    with $Q_{\alpha(x)} \coloneqq \sum_{i\in \textrm{I(x)}} \alpha_i(x) Q_i$, $\{\alpha_i\}_\i$ satisfying \eqref{eq:sliding}\footnote{Without loss of generality, we can take $t\mapsto \alpha_i(x(t))$ almost everywhere continuous (because $\dot x$ in \eqref{eq:system_nou} is), 
    which ensures integrability in \eqref{eq:performance}.}. \hfill $\square$
\end{proposition}
\begin{pf}
For any $\xi \in \R^n$, it holds that  $\dot{v}_{\min}(x,\xi)\coloneqq \lim_{h\rightarrow 0^+} \frac{\vmin(x+h\xi)- \vmin(x)}{h} = \min_{i\in \mathrm{I}(x)} \dot v_i (x,\xi) \leq \dot v_{\sigma(x)} (x,\xi)$ \cite[Lem.~2]{Hu:Ma:Lin:Composite:TAC:2008}; therefore, by \eqref{eq:asmsliding}, along any solution satisfying \eqref{eq:sliding}, it holds, for almost all $t$, that 
\begin{align}\dot{v}_{\min{}} (x,\dot x) &  \leq \textstyle \sum _{i\in \mathrm{I}(x)} \alpha_i(x)  x^\top ( A_i^\top P_i + P_i A_i) x  
\\ & <  -\textstyle \sum _{i\in \mathrm{I}(x)} \alpha_i(x) x^\top Q_i x,
\label{eq:Lyapdecrease}
\end{align}
where the last inequality follows by \eqref{eq:LM_original}, because $x^\top P_i  x \leq x ^\top P_j x$ for $i\in \mathrm{I}(x)$, $j\in \mc{I}$,
and thus
$x^\top (\sum_{j\in \mc{I}} \pi_{j,i} P_j )x \leq 0$ for all $i\in \mathrm{I(x)}$. Stability follows by \eqref{eq:Lyapdecrease} because $\vmin$ is radially unbounded; the inequality \eqref{eq:performance} holds by integrating \eqref{eq:Lyapdecrease} over time, since $\vmin(x(t))\rightarrow 0$ as $t\rightarrow \infty$. \hfill $\blacksquare$
\end{pf}

 While the stability in \cref{prop:LMstability} was  established in \cite{Geromel:Colaneri:Hsu:Solucoes:Brasileiro:2018,Geromel:Colaneri:Bolzern:TAC:2008}, with respect to \cite[Th.~1]{Geromel:Colaneri:Bolzern:TAC:2008} we refined the  guarantee \eqref{eq:performance} to  cope with the possible occurrence of sliding motions. 
We next leverage the result to review $\H_2$ and $\H_\infty$ problems for  switched linear systems. 

\begin{remark}[Solvability of \gls{LM} inequalities]
A sufficient condition for  \eqref{eq:LM_original} to admit a solution is that a convex combination of the matrices $\{A_i\}_{\i}$ is Hurwitz \cite{GeromelColaneri_SIAM2006}, meeting the classical stabilizability condition given, e.g., in \cite{Liberzon2003}. Yet, finding a solution 
is not easy, as the problem is nonconvex --due to the bilinear terms in $\{P_i\}_\i$ and $\Pi$. In fact, in practice relaxations are usually employed to reduce the computational load, e.g., \cite[Th.~4]{GeromelColaneri_SIAM2006}. \hfill $\square$
\end{remark}
%
%
\subsection{$\H_2$ and $\H_\infty$ control}\label{app:LM:performance}

Let us consider a switched linear system 
    \begin{align}
    \label{eq:system_nou2}
    \dot{x}   ={A}_\sigma x + E_\sigma \dist , \quad     z  = C_\sigma x + F_\sigma \dist
\end{align}
where $\dist(t)\in\R^q$ is an exogenous disturbance, $z \in \R^p$ is a performance output, and the matrices $\{E_i\}_\i$ and $\{F_i\}_\i$ measure the influence of the disturbance signal on the state evolution and output, respectively. 
To cope with sliding motions, let us define the \emph{modified} output \begin{align} 
    z_\alpha = \textstyle \sum_{i\in \mathrm{I(x)} } \alpha_i(x) z_i,
\end{align}
with weights $\alpha_i$ as in \eqref{eq:sliding};  note that $z^\alpha$ coincides with $z$ in the absence of sliding mode. We argue that our definition of $z_\alpha$ is very natural to deal with Filippov solutions: as ideal sliding motion approximates fast switching, the performance evaluation should represent all  modes involved (contrarily, $\sigma(x)$ in \eqref{eq:switching_control} would be constant along any sliding trajectory, which is not representative of the behavior of a real system when chattering occurs).

We  study the performance of \eqref{eq:system2} with respect to the channel $(\dist,z_\alpha)$. In particular, assuming  that $\sigma(x)$ is a stabilizing state-feedback controller and that $x(0)=0$, we consider the following performance indices:
\begin{itemize}[leftmargin = *]
    \item \emph{$\mc{H}_2$ index}: Let $F_i=0$, $E_i = E$  for all $\i$;  denote by $x_k:\R_{\geq 0} \rightarrow \R^{n}$ and $z^k_\alpha:\R_{\geq 0} \rightarrow \R^{p}$  state and modified output trajectories generated with  the the disturbance $\dist_k(t)\coloneqq e_k \delta(t)$ (i.e., with zero disturbance and $x(0)= E e_k$).
    Then
    \begin{align}\label{eq:H2}
        J_2 \coloneqq  \sum_{k=1} ^q \|z^k_\alpha\|_2^2 \ ;
      \end{align}
     \item  \emph{$\mc{H}_\infty$ index}: Let 
     $\bar x:\R_{\geq 0} \rightarrow \R^{n}$ and $\bar z_\alpha:\R_{\geq 0} \rightarrow \R^{p}$  state and modified output trajectories generated with 
     an arbitrary disturbance $\bar{\dist}\in \mc{L}_2$. We define
        \begin{align}\label{eq:Hinf}
        J_\infty =  \sup_{0 \neq \bar \dist \in \mc{L}_2} \frac{\|\bar z_\alpha\|_2^2}{\|\bar \dist\|_2^2}.
     \end{align}
\end{itemize}
Intuitively, if $N=1$, the definitions recover the standard $\H_2$ and $\H_\infty$ performance indices for \gls{LTI} systems. For switched linear systems, the quantities $J_2$ and $J_\infty$ were defined analogously in the literature, but in terms of the output $z$ \cite{Geromel:Colaneri:Bolzern:TAC:2008,Deaecto:Geromel:Hinf:ASME:2010}. In fact, these works only consider Carathéodory solutions. The following two propositions fill this gap, by considering the  modified output $z_\alpha$ and by refining the proof of \cite[Th.~3]{Geromel:Colaneri:Bolzern:TAC:2008}, \cite[Th.~2]{Deaecto:Geromel:Hinf:ASME:2010} to account for the presence of sliding motions. 

\setpropositiontag{A2}
\begin{proposition}[$\H_2$ control]\label{prop:H2}
    If there exist  $\{P_i\succ 0\}_{\i}$ and  $\Pi=[\pi_{i,j}]_{i,j \in \mc{I}}  \in \M$ satisfying the Lyapunov-Metzler inequalities \eqref{eq:LM_original} with $\{Q_i = C_i^\top C_i\}_\i$, 
    then the switched feedback control law \eqref{eq:switching_control} makes $x^*=0$ globally asymptotically stable for the system \eqref{eq:system_nou2} with $E_i= E$ for all $\i$, and ensures that $J_2 < \min_{\i} \operatorname{tr}(E^\top P_i E  )$. \hfill $\square$
\end{proposition}
\begin{pf}
\cref{prop:LMstability} implies stability  and that  $J_2  = \sum_{k=1}^q \int_{t=0} ^\infty  x_k^\top Q_\alpha x_k \  \text{d}t
     < \textstyle \sum_{k=1}^q \min_{{\i}}   (Ee_k)^\top P_i (Ee_k)
      \leq  
     \min_{{\i}}  
     \sum_{k=1}^q    (Ee_k)^\top P_i (Ee_k)
$; the conclusion follows by definition of $e_k$. \hfill $\blacksquare$
\end{pf}

\setpropositiontag{A3}
\begin{proposition}[$\H_\infty$ control]\label{prop:Hinf}
If there exist  $\{P_i\succ 0\}_{\i}$, a scalar $\rho >0$ and  $\Pi=[\pi_{i,j}]_{i,j \in \mc{I}}  \in \M$ such that  
    \begin{align}\label{eq:Hinfinequality}
     \begin{bmatrix}
           A_i^\top P_i +P_i A_i + 
          \textstyle \sum_{j\in \mc{I}} \pi_{i,j} P_j & P_i E_i & C_i^\top
          \\
          \star & -\rho I & F_i^\top  \\
          \star & \star & -I
     \end{bmatrix} \prec 0,
 \end{align}
  then the switched feedback control law \eqref{eq:switching_control} makes $x^*=0$ globally asymptotically stable for the system \eqref{eq:system_nou2} and ensures that $J_\infty < \rho $. \hfill $\square$
\end{proposition}
\begin{pf} Stability holds by \cref{prop:LMstability} and negative definiteness of the upper-left block. As in the proof of \cref{prop:LMstability}, we have for almost all $t \geq 0$ \begin{align*}
    \dot v_{\min}(x,\dot x) &  \leq  \dot v_{\sigma(x)}\left( x, \textstyle \sum_{i\in \mathrm{I}} \alpha_i(x) (A_i x+E_i \dist) \right )
    \\
     & \overset{\textnormal{(a)}}{\leq} \textstyle \sum_{i\in \mathrm{I}(x)}\alpha_i(x) \dot v_i(x,A_ix+E_i \dist)
     \\
    & \overset{\textnormal{(b)}}{<
    } \textstyle \sum_{i\in \mathrm{I}(x)}\alpha_i(x) \left(- z_i^\top z_i + \rho \dist^\top \dist \right) 
\end{align*} 
where (a) is the analogous of condition \eqref{eq:asmsliding} (see \cref{foot:slidingcondition}) and (b) follows via a Schur complement argument by \eqref{eq:Hinfinequality} (see \cite[Eq.~14]{Deaecto:Geromel:Hinf:ASME:2010}). The result follows by integrating, since $\vmin(x(0)) = 0$ and  $\vmin(x(t))\rightarrow 0$ as $t\rightarrow \infty$. \hfill $\blacksquare$
\end{pf}

\section{Proofs}
\vspace{-0.8em}
\subsection{Proof of \cref{lem:strictnonstrict}}\label{app:lem:strictnonstrict}
\vspace{-1em}
We only need to show that feasibility of \eqref{eq:LM1} implies feasibility of \eqref{eq:LM2}. Let $\bar S =(\{\bar P_i\}_{\i},\{\bar K_i\}_{\i}, \bar \Pi,\bar Q)$ be a solution for \eqref{eq:LM1}. We prove that, for all $\i$,  the set $\mc{C}_i^\textrm{cl}\coloneqq\{ (A_i+B_i \bar{K}_i)^\top \mid  (A_i,B_i) \in \C_i \} $ is compact; then  there exists a solution to \eqref{eq:LM2} in any neighborhood of $\bar S$ due to the strict inequality in \eqref{eq:LM1}.
\newline\indent
Since the equation in  \eqref{eq:compatibility_definition} is affine, we can write $\C_i=\{(A_i,B_i)\mid \begin{bmatrix}
  A_i & B_i
\end{bmatrix}^\top =E_i^+ V_i +\begin{bmatrix}
  A_i^0 & B_i^0
\end{bmatrix}^\top,  V_i \in \mc{V}_i, (\bar A_i^0,\bar B_i^0)\in  \C_i^0 
\}$, with $E_i \coloneqq \begin{bmatrix}
  X_i^\top & U_i^\top
\end{bmatrix}$, $E^+$ its pseudoinverse, $\mc{V}_i \coloneqq (\{\dot X_i^\top \} + \mc\noiseset_i^\top  ) \cap \{E_i\alpha \mid  \alpha \in \R^{ (n+m) \times n} \}$, $\C_i^0 = \{(A_i^0,B_i^0)\mid A_i^0X_i +B_i^0 U_i=0\}$ (in simple terms, $E_i^+ V_i$ is a particular solution and $\C_i^0$ are the --disturbance independent-- homogeneous solutions). 
\newline\indent
Consider any $(\bar A_i^0,\bar B_i^0)\in \C_i^0$. We define $\bar A_i^\prime \coloneqq (\bar A_i^0+ \bar B_i^0 \bar{K}_i)^\top\bar A_i^0$, $ \bar B_i^\prime \coloneqq (\bar A_i^0+ \bar B_i^0 \bar{K}_i)^\top\bar B_i^0)$ and note that $(\bar A_i^{\prime}, \bar B_i^{\prime}) \in \C_i^0$.  
We claim that
the symmetric matrix $M \coloneqq( (\bar A_i^\prime+\bar B_i^\prime \bar{K}_i)^\top\bar P_i + \bar P_i(\bar A_i^\prime+\bar B_i^\prime \bar{K}_i) )$ is nilpotent, hence equals $0$. If not, take any $V_i \in \mc V _i$, let $\begin{bmatrix} A_i^\beta & B_i^\beta 
  \end{bmatrix}^\top \coloneqq E_i^+ V_i +\beta \begin{bmatrix} \bar A_i^\prime &\bar B_i^\prime  \end{bmatrix}^\top$,
and note that $ (A_i^\beta,B_i^\beta) \in \C_i $ for all $\beta \in \R$; yet this pair violates \eqref{eq:LM1} for $\beta>0$ or $\beta<0$ large enough, providing a contradiction. 
This also means that the symmetric matrix 
$ \bar A_i^\prime+\bar B_i^\prime \bar{K}_i = (\bar A_i^0+ \bar B_i^0 \bar{K}_i)^\top (\bar A_i^0+ \bar B_i^0 \bar{K}_i)$ is nilpotent, hence equals $0$ (if not, let $\R \ni \lambda \neq 0$ and $\R^n \ni m \neq 0$ be an eigenvalue-eigenvector couple, and note that $ m^\top M m  = 2 \lambda m^\top P_i  m = 0 $ contradicts $P_i \succ 0$). 
Therefore, we finally have $\bar A_i^0+\bar B_i^0 \bar{K}_i = 0$. 
\newline\indent
We conclude that 
$\mc{C}_i^\textrm{cl} = \{ \begin{bmatrix} I & \bar K_i^\top 
\end{bmatrix} E_i^+ V_i  \mid V_i \in \mc V_i \}$. The proof follows because $\mc \noiseset_i$ is compact (due to $\Phi_{2,2}^i\prec 0$ in \cref{asm:disturbancemodel}), and so must be $\mc V_i$ and in turn $\mc{C}_i^{\textrm{cl}}$. \hfill $\blacksquare$

\subsection{Proof of \cref{lem:existenceofspecialsol}}\label{app:lem:existenceofspecialsol}
\vspace{-1em}
Let   $M_i \coloneqq  {A_i^{\textnormal cl}}^\top P + P A_i^{\textnormal cl}$. With $\Pi$ as in the statement, $P_i = N P+\mu^{-1} \lambda_ i M_i$,  $Q=0$, and recalling that $\sum_{j\in \I} \pi_{j,i} P=0$,  the left-hand side of \eqref{eq:LM1}  is
\begin{align*}
    \Acli^\top (NP+\tilde \mu \lambda_i M_i)+ \star + \textstyle \sum_{j\in \mc{I}\backslash \{i\} }  {\tilde \lambda_i}  (  \lambda_j M_j -  \lambda_i M_i),
\end{align*}
($\star$ is the transpose of the first addend, and we recall that $\tilde \lambda = \lambda^{-1}$, $\tilde \mu = \mu ^{-1}$), which is negative definite for $\mu$ large enough, as taking its limit $\mu \rightarrow \infty$ gives
\begin{align}
    \Acli^\top P + P \Acli + \lambda_i^{-1} \textstyle \sum_{j \in \mc{I}\backslash \{i\} }\lambda_j  M_j \prec 0
\end{align}
where the inequality is  \eqref{eq:quadratic}. Thus we  constructed a solution to \eqref{eq:LM1} based on \eqref{eq:quadratic}; a solution to \eqref{eq:LM2} with the same $\Pi$, $P_i$'s, and some $Q\succ 0$ then exists as per \cref{lem:strictnonstrict}. \hfill $\blacksquare$

\vspace{-0.7em}


%
%

\bibliographystyle{abbrv}
\bibliography{DDC_switched}


\end{document}